\documentclass{amsart}
\usepackage{latexsym,amsmath,amsthm,amssymb}
\usepackage{color}

\newtheorem{theorem}{Theorem}[section]
\newtheorem{lemma}[theorem]{Lemma}
\newtheorem{corollary}[theorem]{Corollary}
\newtheorem{proposition}[theorem]{Proposition}

\theoremstyle{remark}
\newtheorem{example}[theorem]{Example}
\newtheorem{remark}[theorem]{Remark}

\theoremstyle{definition}
\newtheorem{definition}[theorem]{Definition}

\def\Z{\mathbb{Z}}
\def\Nat{\mathbb{N}}

\newcommand{\dR}{\ensuremath{\mathbb{R}}} 
\newcommand{\R}{\dR}

\newcommand{\Exp}{{\rm I\!E}}

\begin{document}

\title{On continuity equations  in infinite dimensions with non-Gaussian reference measure}

\author{Alexander~V.~Kolesnikov}
\address{ Higher School of Economics, Moscow,  Russia}
\email{Sascha77@mail.ru}

\vspace{5mm}
\author{Michael R{\"o}ckner}
\address{
Universit{\"a}t Bielefeld, Bielefeld, Germany
}
\email{roeckner@math.uni-bielefeld.de}

\thanks{
This study was carried out within ''The National Research University Higher School of Economics`` Academic Fund Program in 2012-2013, research grant No. 11-01-0175. The first author was supported by RFBR projects 10-01-00518, 11-01-90421-Ukr-f-a.
Both authors would like to thank the DFG for financial support through CRC 701.
}
\subjclass[2000]{} 

\vspace{5mm}

\begin{abstract}
Let $\gamma$ be a Gaussian measure on a locally convex space and $H$ be the corresponding Cameron-Martin space.
It has been recently shown by L.~Ambrosio and A.~Figalli  that  the linear first-order PDE
$$
\dot{\rho} +  \mbox{div}_{\gamma} (\rho \cdot {b})  =0, \ \ \rho|_{t=0} = \rho_0,
$$
where $\rho_0 \cdot \gamma $ is a probability measure, admits a  weak solution, in particular, under the following assumptions:
$$
\|b\|_{H} \in L^p(\gamma), \ p>1, \ \ \
 \exp\bigl(\varepsilon(\mbox{\rm div}_{\gamma} b)_{-} \bigr)
\in
L^1(\gamma).
$$
Applying  transportation of measures via triangular maps 
 we prove a similar result for  a large class of non-Gaussian probability measures $\nu$ on $\R^{\infty}$, under the main assumption
that
$\beta_i \in \cap_{n \in \Nat} L^{n}(\nu)$ for  every $i \in \Nat$, where $\beta_i$
is the logarithmic derivative of $\nu$ along the coordinate $x_i$.
 We also show uniqueness of the solution 
for a wide class of measures. This class includes  uniformly log-concave Gibbs measures and
 certain product measures.
\end{abstract}

\maketitle

\section{Introduction}

In this paper we study  infinite-dimensional continuity equations
\begin{equation}
\label{measures-ev}
\dot{\mu} + \mbox{div}(\mu \cdot b) =0, \ \mu_0= \zeta,
\end{equation}
where $\mu = \mu_t(dx)$, $t \ge 0$, is a curve of probability measures on $\mathbb{R}^{\infty}$
equipped with the product $\sigma$-algebra induced by the Borel $\sigma$-algebra on $\mathbb{R}$ and $b : \mathbb{R}^\infty 
\to \mathbb{R}^{\infty}$. Furthermore, $\dot{\mu} = \frac{\partial}{\partial t} \mu$, $\mbox{div}$ is meant in the sense of distributions and $\zeta$ is a probability measure on $\mathbb{R}^{\infty}$ serving as the initial datum. One approach to solve
equation (\ref{measures-ev}) is to choose a reference measure $\nu$ and search for solutions for  (\ref{measures-ev})
with $\zeta = \rho_0 \cdot \nu$ which are of the form $\mu_t(dx) = \rho(t,x) \cdot \nu(dx)$.
Then (\ref{measures-ev}) can be written as  
\begin{equation}
\label{main-eq}
 \dot{\rho} +  \mbox{div}_{\nu} (\rho \cdot {b})  =0, \ \rho(0,x) = \rho_0,
\end{equation}
where $\mbox{div}_{\nu}$ is the divergence with respect to $\nu$, i.e. $(-1)$ times the adjoint of the gradient operator
on $L^2(\mathbb{R}^{\infty},\nu)$. We stress that the choice of the reference measure (even in the finite-dimensional case,
where $\mathbb{R}^{\infty}$ is replaced by $\mathbb{R}^d$) is at our disposal and should be made depending on $b$.
For instance, in the finite-dimensional case $b$ might be in a weighted Sobolev class with respect to some measure $\nu$ absolutely continuous with respect to Lebesgue measure, but not weakly differentiable with respect to Lebesgue measure itself.
Then one should take $\nu$ to be the measure for which the components of $b$ are in $W^{1,2}(\nu)$. So, singularities 
of $b$ will thus be compensated by the zeros of the Lebesgue density of $\nu$.

Likewise in the infinite dimensional case of $\mathbb{R}^{\infty}$, where one usually takes a Gaussian measure 
as reference measure, since they are best studied. However, in many cases this is not the best choice, for similar reasons as we have just seen, in the case of $\mathbb{R}^d$. For instance, there are interesting examples (presented in Section 7.2 below),
 where the reference measure should be taken to be a Gibbs measure, whose energy functional can be ''read off'' the given map $b$ which determines equation (\ref{measures-ev}), respectively (\ref{main-eq}).
 
The key point, that for such reference measures we can identify conditions so that (\ref{main-eq}) has a solution and/or
that this solution is unique, lies in the fact that many probability measures on $\mathbb{R}^d$ are images of Gaussian measures under so-called triangular mappings which turn out to have sufficient regularity in many concrete situations.  Therefore, we
 can reduce existence and uniqueness questions (\ref{measures-ev}), respectively (\ref{main-eq}) to the case of a Gaussian reference measure, studied in \cite{AF} and \cite{FangLuo}.

To explain this and also to review a bit the history of the problem, let us return to equation (\ref{main-eq}) and recall that the associated Lagrangian flow has the form

\begin{equation}
\label{Lagr}
\dot{X}(t,x) = b(X), \ X(0,x)=x.
\end{equation}
A finite-dimensional theory of equations  (\ref{measures-ev}) and (\ref{main-eq}) for  weakly differentiable drifts $b$
has been  deeply developed in a recent series of papers by L.~Ambrosio. G.~Crippa, C.~De Lellis, G.~Savar{\'e}, A.~Figalli  and others (see \cite{Ambrosio05} and the references therein).
This theory works under quite general assumptions and includes, in particular, existence and uniqueness results
for  BV (bounded variation) vector fields.

Relatively little is known, however, in the infinite-dimensional setting.
The first results in this direction
have been obtained by A.B.~Cruzeiro   \cite{Crus}, V.I.~Bogachev and E.~Mayer-Wolf \cite{BW}.
The starting point for us was the paper \cite{AF}, where some
finite-dimensional techniques (including the Di Perna-Lions theory of renormalized solutions) have been generalized
 to the infinite-dimensional Gaussian case.
 Other recent developments can be found in Di Perna-Lions \cite{DPL1}, \cite{DPL2}, Ambrosio-Figalli \cite{flow2009}, Le Bris-Lions \cite{BL}, Fang-Luo \cite{FangLuo}, Bogachev-Da Prato-Shaposhnikov-R{\"o}ckner \cite{BDRSh}.

We stress that the uniqueness of the solution is a more difficult problem compared to the
existence.  The latter can be established under quite broad assumptions (see, for instance, \cite{BDRSh} for the apparently most general results about existence).
The uniqueness  proof obtained in \cite{AF} relies very strongly on the Gaussian framework. An important
technical point was  smoothing by the Ornstein-Uhlenbeck semigroup which behaves
very nicely with respect to many natural operations on the Wiener space (divergence, projections, conditional expectations, differentiation etc.).
The absence of such a nice smoothing operator seems to be the main difficulty when one tries to solve (\ref{main-eq})
for  non-Gaussian reference measures.

In this paper we prove an existence result for the case of  reference measures $\nu$ on $\R^{\infty}$ with
logarithmic derivatives integrable in any power. We also show uniqueness
for a wide class of product measures, including 
 log-concave ones. Another uniqueness result is proved for a class of uniformly log-concave Gibbs measures.

Our approach relies on the mass transportation method. The general scheme works as follows. Instead of directly solving
 (\ref{main-eq}) we consider a mass transportation mapping
$T : \R^{\infty} \to \R^{\infty}$ pushing forward the standard Gaussian measure  $\gamma$ onto $\nu$:
$\nu = \gamma \circ T^{-1}$.
If $\nu_t = \rho_t \cdot \nu$ is the solution to (\ref{main-eq}), then the  family of measures $\gamma_t
= \nu_t \circ S^{-1} $ with $S = T^{-1}$
solves  the continuity equation for the new vector field
\begin{equation}
\label{cdef}
c = DT^{-1} \cdot b(T),
\end{equation}
here $D$ denotes total derivative.
Applying (slightly generalized) existence and uniqueness results for the Gaussian case from \cite{AF}, we get a
solution $\gamma_t$ of  the equation associated to the vector field $c$ and transfer it back, i.e.
$\nu_t = \gamma_t \circ T^{-1}$.

The main advantage of this approach is that the
divergence operator commutes with $T$:
$$
\mbox{div}_{\gamma} c   = \bigl[  \mbox{div}_{\nu} b \bigr] \circ T.
$$
Hence the
crucial assumptions on $\mbox{div}_{\gamma} c$ can be directly transferred to $\mbox{div}_{\nu} b$.
On the other hand, assumptions on integral norms  of $c$ and $D c$ impose some restrictions on Sobolev norms of $T$ and $S=T^{-1}$.
To prove the corresponding a-priori bounds is the main technical difficulty of our approach.

Note that we are free to choose any type of transportation mappings provided they have sufficient regularity.
In this paper we deal with triangular mass transportation. A short discussion about the optimal transportation approach can be found  in the very last section of this paper.
The advantage of these
mappings is their simple form. Even in the infinite-dimensional case they
have essentially finite-dimensional structure.
We obtain some Sobolev estimates on $S$ and deduce from them the existence result for (\ref{main-eq}).
The key estimate for triangular mappings applied in this paper looks as follows.
Let
$S = \sum_i S_i \cdot e_i$ be the  triangular mapping pushing forward the measure $\nu$ onto the standard Gaussian measure $\gamma$. Then
  $$
\int \| \partial_{x_j} S \|^2_{l^2} \ d\nu  =
\sum_{i \ge j}
\int \bigl( \partial_{x_j} S_i \bigr)^2 \  d\nu \le
  \int   \beta_j^2 \   d\nu.
$$
Here $\beta_j$ is the logarithmic derivative of $\nu$ along $x_j$.
For more details on  triangular mappings see   \cite{B2006}.

The paper is organized as follows. In Section 2 we prove an extension of the
results from \cite{AF}. In particular, we weaken some assumptions in  \cite{AF} by introducing
 a slightly weaker notion of  solution (see Remark \ref{04.03}).  In Section 3 we establish  Sobolev estimates for triangular mappings.
In Section 4 we prove the key technical relations between transport equations and mass transfer.
 The existence result is  proved in Sections 5.
Sections 6-7  deal with the uniqueness
in the product and Gibbsian case.
In particular, we prove a uniqueness result  for log-concave   Gibbs measures
with the following formal  Hamiltonian
$$
\sum_{i=1}^{\infty} V_i(x_i) + \sum_{i,j=1}^{\infty} W_{i,j}(x_i,x_k).
$$
In {the Appendix} we briefly discuss the approach via optimal transportation mappings and the finite-dimensional case. 
In particular, we prove an existence and uniqueness theorem for a broad class of log-concave measures under "dimension-free" assumptions.   
Furthermore, in Example \ref{fin-dim-ex} we give an example in the finite-dimensional case, for which our result (see Theorem \ref{logconcave})
implies existence and uniqueness for (\ref{main-eq}), where $b: \mathbb{R}^{\infty} \to \mathbb{R}^{\infty}$
is not BV (hence the results of \cite{AF}, \cite{Ambrosio04}  are not applicable).

{\bf Notations:}  Throughout the paper $p^*$ is the dual  numbers to $p \in [1,\infty[$:
$\frac{1}{p}+\frac{1}{p^*}=1.$  
We denote by  $\mathcal F_n$  the $\sigma$-algebra generated by the projection 
$P_n(x) = (x_1,\cdots, x_n) $
and by $\mathbb{E}^{\mathcal{F}_n}_{\nu}$ the corresponding conditional expectation.
Everywhere below $\| \cdot \|$  means the standard $l^2$-norm (finite and infinite dimensional).
We  denote by $\nabla$ and $D^2$   the derivatives of first and second order along
$H = l_2$ respectively. For every linear operator $A: l_2 \to l_2$  the notation $\| A \|$  means the standard operator norm and 
 $\| A \|_{HS} = \sqrt{\mbox{Tr} (A^* A)}$ the Hilbert-Schmidt norm.
The time derivative of a function $f$ is denoted by $\dot{f}$.
We fix the standard  orthogonal basis in $\mathbb{R}^{\infty}$
consisting of vectors
$e_i =(\delta_{ij})_{j \in \mathbb{N}}$.
We use the word ''positive'' in the sense of ''strictly positive'' (i.e. ''$>0$''), otherwise we say ''nonnegative''

{\bf Acknowledgement.} 
We thank the referee for the careful reading and many suggestions which help us to make significant improvements of the paper.

\section{The Gaussian case}

In this paper we use the following core of smooth cylindrical functions: $\mathcal{C}$ is the linear span of all infinitely  differentiable functions 
$\varphi(x_1, \cdots, x_n)$  depending on a finite number of coordinates and having a  compact (considered as functions on $\mathbb{R}^n$) support.
\begin{remark}
\label{04.03}
\begin{itemize}
\item[(i)] The use of  functions of the form $\varphi(x_1, \cdots, x_n)$, $\varphi \in C^{\infty}_0(\R^n)$,
 is natural for $\R^{\infty}$, but differs from the standard core in the Gaussian case, where
$\varphi$ usually depends on a finite collection of   measurable functionals $X_{h_i}$, $h_i \in H$, which are $\mathcal{N}(0, \|h_i\|^2)$-distributed.
\item[(ii)]
Clearly, $\mathcal{C}$ separates the points of $\mathbb{R}^{\infty}$. Furthermore, a simple monotone class argument shows that $\mathcal{C}$ is dense in any $L^p(\nu)$, $p \in [1, \infty)$
and any finite measure $\nu$ on $\mathbb{R}^{\infty}$.
\end{itemize}
\end{remark}

Let $\nu$ be a probability measure on $\mathbb{R}^{\infty}$.
We say that a mapping $b: \R^{\infty} \to \R^{\infty}$
has  divergence $\mbox{div}_{\nu} b \in L^1(\nu)$
if the following relation holds for every $\varphi \in \mathcal{C}$:

\begin{equation}
\label{div-def}
\int \mbox{div}_{\nu} b \
\varphi \ d \nu
=
- \int \langle b,  \nabla \varphi \rangle \ d \nu.
\end{equation}

For an account in infinite-dimensional analysis on spaces with differentiable measures the readers are referred to
\cite{B2006}, \cite{B2008}.

We study (\ref{main-eq}),
where $\rho=\rho(t,x)$ is a family of probability densities with respect to $\nu$
with  initial condition $\rho(0,\cdot)=\rho_0$, i.e. we are looking for  solutions $\rho(t,x)$
given as densities of a family of probability measures $\mu_t(dx) = \rho(t,x) \cdot \nu(dx).$

\begin{definition}
\label{main-def}
We say that $\rho$ is a solution of (\ref{main-eq}) for $t \in [0,T]$  with  initial value $\rho_0$ if for every $\varphi \in \mathcal{C}$ and $t \in [0,T]$
one has
\begin{equation}
\label{distrib}
\int \varphi \rho(t,x)\ d \nu
=
\int \varphi \rho(0,x) \ d \nu
+
\int_{0}^{t} \int \langle {b}, \nabla \varphi \rangle  \rho(s,x) \ d\nu \ ds.
\end{equation}
\end{definition}

\begin{remark}
The solution in the finite-dimensional case is defined in the same way.
\end{remark}

\begin{remark}
We note that the existence of the right-hand side is not obvious because it is not clear a-priori that $ \langle {b}, \nabla \varphi \rangle  \rho(s,x) \in L^1(I_{[0,t]} ds  \times \nu)$. Nevertheless, we will see in the following Lemma that this is indeed the case if 
$c$ defined in (\ref{cdef}) satisfies some natural assumptions.
\end{remark}

The following result has been proved by Ambrosio and Figalli
in \cite{AF} (Theorem 6.1) for $\rho_0 \in L^{\infty}(\gamma)$.
The proof of this result is the same and so we omit it here.

\begin{lemma}
\label{AFlemma0}
Consider the standard Gaussian measure $\gamma$ on $\R^d$.
Let $\|c \|\in L^p(\gamma)$, $p>1$ and $
\| \exp\bigl(\varepsilon(\mbox{\rm div}_{\gamma} c)_{-} \bigr)\|_{L^1(\gamma)} < \infty$
for some $\varepsilon >0$.
Then for any $\rho_{0} \in L^{q'}(\gamma)$ with $q'>q = \frac{p}{p-1}=p^*$  there exists $T=T(\varepsilon,p,q')>0$ such that the
equation
$$
\dot{\rho} + \mbox{\rm div}_{\gamma} \bigl( c \cdot \rho \bigr) =0
$$
admits a solution $\rho$ on $[0,T]$
satisfying $\sup_{t \in [0,T]} \|\rho_{t}\|_{L^q(\gamma)} <  C$, { where 
$C$ depends on the $L_p$-norms mentioned in the assumption of the Lemma and does not depend on
the dimension of the space $d$}.
\end{lemma}

Let us give the idea how to control the $L^p$-norms of $\rho_t$ via $ \mbox{div}_{\gamma} c$ needed in the proof of Lemma \ref{AFlemma0}. Below we set  for brevity
$$
\rho_t = \rho(t,\cdot), \ \ \mbox{and} \ \ X_t :=X(t, \cdot)
$$
(see (\ref{Lagr})).
The well known    change of variables formula for the mapping $x \to X_t(x)$
is given by the Liouville formula:
$$
\rho_t (X_t)= \rho_0 \cdot \exp \Bigl( - \int_{0}^{t} \mbox{div}_{\gamma} c(X_r) \ dr\Bigr) = \rho_s(X_s) \exp \Bigl( - \int_{s}^{t} \mbox{div}_{\gamma} c(X_r) \ dr\Bigr) .
$$
{Applying Jensen's and H{\"o}lder inequalities one gets for any $q \ge 1$ 
\begin{align*}
\int \rho^q_t \ d \gamma 
& = 
\int \rho^{q-1}_t(X_t) \rho_0 \ d \gamma 
=
\int \rho^{q}_0  \exp \Bigl( - \int_{0}^{t}(q-1) \mbox{div}_{\gamma} c(X_r) \ dr\Bigr) \ d \gamma 
\\&
\le 
\frac{1}{t} \int_0^t \int  \rho^{q}_0  \exp \Bigl( - t(q-1) \mbox{div}_{\gamma} c(X_r) \Bigr) \ d \gamma \ dr
\\&
\le \Bigl( \frac{1}{t} \int_0^t  \int \rho_0^{{q'-1}} \rho_0 \ d \gamma dr \Bigr)^{\frac{q-1}{q'-1}}
 \Bigl( \frac{1}{t} \int_0^t  \int   \exp \Bigl( - t \frac{(q'-1)(q-1)}{q'-q} \mbox{div}_{\gamma} c(X_r) \Bigr) \rho_0 \ d \gamma dr\Bigr)^{\frac{q'-q}{q'-1}}
\\& =   \bigl(  \int \rho_0^{{q'}} d \gamma\bigr)^{\frac{q-1}{q'-1}}  \Bigl( \frac{1}{t} \int_0^t  \int   \exp \Bigl( - t \frac{(q'-1)(q-1)}{q'-q} \mbox{div}_{\gamma} c \Bigr) \rho_r \ d \gamma dr \Bigr)^{\frac{q'-q}{q'-1}}
\\& \le 
   \bigl(  \int \rho_0^{{q'}} d \gamma\bigr)^{\frac{q-1}{q'-1}}  \Bigl( \frac{1}{t} \int_0^t  \int  \rho^q_r \ d \gamma dr \Bigr)^{\frac{q'-q}{q(q'-1)}}
\Bigl( \frac{1}{t} \int_0^t  \int   \exp \Bigl( - t \frac{q(q'-1)}{q'-q} \mbox{div}_{\gamma} c \Bigr) \ d \gamma dr \Bigr)^{\frac{(q'-q)(q-1)}{q(q'-1)}}.
\end{align*}
Thus one gets that $\Lambda(t) = \int_0^t \int \rho^q_r \ d \gamma \ dr$
satisfies 
$$\Lambda' \le C (\Lambda/t)^{\delta}$$
with $$\delta = {\frac{q'-q}{q(q'-1)}}, \ \ C = 
   \bigl(  \int \rho_0^{{q'}} d \gamma\bigr)^{\frac{q-1}{q'-1}}  \Bigl(   \int   \exp \Bigl( - t \frac{q(q'-1)}{q'-q} \mbox{div}_{\gamma} c \Bigr) \ d \gamma  \Bigr)^{\frac{(q'-q)(q-1)}{q(q'-1)}}.$$
Note that $C < \infty$ if $t  \le \varepsilon \frac{q'-q}{q(q'-1)}$.

Integrating inequality $\Lambda' \le C (\Lambda/t)^{\delta}$ one gets
$\Lambda \le C^{\frac{1}{1-\delta}} t$, hence $\Lambda' \le C^{\frac{1}{1-\delta}}$. Thus we obtain
\begin{equation}
\label{lq-exp}
\int \rho^q_t \ d \gamma \le  C^{\frac{1}{1-\delta}}. 
\end{equation}} 
\begin{lemma}
\label{AFlemma}
Let $c$ satisfy the assumptions of Lemma \ref{AFlemma0}
and $f$ be a bounded Lipschitz function:
$$
|f|_{{Lip}}= \sup_{x \ne y} \frac{|f(x)-f(y)|}{\|x-y\|} < \infty.
$$ 
Then {any} solution $\rho_t$ obtained in  Lemma \ref{AFlemma0}
satisfies the following property:

1) If, in addition,  
 $$\mbox{\rm{div}}_{\gamma} c \in L^{N}(\gamma),  \ \ \ \  \|c\| \in  L^{p'}(\gamma) \  \ \ \ \mbox{for some} \
N > q^* , \ p'>p,
$$
then there exist positive constants $C,\delta$ depending on 
$$
p, q, p', q', \ \  \| c \|_{L^{p'}(\gamma)}, \ \ \| \mbox{\rm{div}}_{\gamma} c\|_{L^{N}(\gamma)}, \ \  \| \exp\bigl(\varepsilon(\mbox{\rm div}_{\gamma} c)_{-} \bigr)\|_{L^1(\gamma)}
$$
$$
\|\rho_0\|_{L^{q'}(\gamma)}, \ \sup |f|, \ |f|_{Lip}
$$
 such that 
\begin{equation}
\label{Lp-Hoeld-delta}
\Bigl| \int f \rho^{1+\delta}_t \ d \gamma - \int f \rho^{1+\delta}_s \ d \gamma \Bigr| \le C |t-s|, \ \mbox{for all} \ t, s \in [0,T].
\end{equation}

2) Without any extra assumption
there exists a positive constant $C$ depending on 
$$
p, q', \ \  \| c \|_{L^{p}(\gamma)}, \ \  \| \exp\bigl(\varepsilon(\mbox{\rm div}_{\gamma} c)_{-} \bigr)\|_{L^1(\gamma)}, 
\|\rho_0\|_{L^{q'}(\gamma)},  \ |f|_{Lip}
$$
 such that 
\begin{equation}
\label{Lp-Hoeld}
\Bigl| \int f \rho_t \ d \gamma - \int f \rho_s \ d \gamma \Bigr| \le C |t-s|, \ \mbox{for all} \ t, s \in [0,T].
\end{equation}
\end{lemma}

\begin{proof}
We prove only 1) because the proof of 2) is easier and follows the same line.
 In the same way as in \cite{AF} we reduce the proof to the case when
$X(t,x)$ is a globally defined smooth solution to $\dot{X}=c(X), \ X(0,x)=x$.
We apply the  change of variables formula for the mapping $x \to X_t(x)$.
Let $s<t$, $\delta>0$ and $f$ be a bounded Lipschitz function.
\begin{align*}
\int & \rho^{1+\delta}_t f \ d \gamma =  \int \rho^{\delta}_t(X_t) f(X_t) \rho_0 \ d \gamma \\& = \int  \rho^{\delta}_s(X_s) f(X_t) \exp \Bigl( - \delta \int_{s}^{t} \mbox{div}_{\gamma} c(X_r) \ dr\Bigr) \rho_0 \ d \gamma
\\&
=  \int \rho^\delta_s(X_s) f(X_t)\ d \gamma + \int  \rho^\delta_s(X_s) f(X_t) \bigl[\exp \Bigl( - \delta \int_{s}^{t} \mbox{div}_{\gamma} c(X_r)\ dr\Bigr) -1 \bigr] \rho_0 \ d \gamma
\\& =
\int \rho^{1+\delta}_s  f \ d \gamma + \int \rho^{\delta}_s(X_s) \bigl( f(X_t)  - f(X_s) \bigr) \ \rho_0\ d \gamma 
\\&  + \int  \rho^{\delta}_s(X_s) f(X_t) \bigl[\exp \Bigl( - \delta \int_{s}^{t} \mbox{div}_{\gamma} c(X_r)\ dr\Bigr) -1 \bigr] 
\rho_0 \ d \gamma.
\end{align*}
Here we use that $\rho_s \cdot \gamma$ is the image of $\rho_0 \cdot \gamma$ under $x \to X(t,x)$.
Note that $|e^{-t}-1| \le u(t)$, where $u(t) = e^{\max\{-t,0\}}|t|$. Since $u$ is convex one can apply the Jensen inequality. Then the last term in the right-hand side of the above inequality can be estimated by
$$
\frac{\sup |f|}{t-s} \int  \rho^{\delta}_s(X_s) \int_{s}^{t}   u\bigl( \delta (t-s)\mbox{div}_{\gamma} c(X_r)\bigr) \ dr \ \rho_0 \ d \gamma.
$$
The latter can be estimated by 
\begin{align*}
 \sup |f| \Bigl[  \int  \rho^{q_1 \delta+1}_s \ d \gamma \Bigr]^{\frac{1}{q_1}}  &  \Bigl[ \frac{1}{t-s} \int\int_{s}^{t}   u^{q^*_1}\bigl( \delta (t-s) \mbox{div}_{\gamma} c\bigr)  \ \rho_r  \ dr\ d \gamma \Bigr]^{\frac{1}{q^*_1}}
\\&
\le
\sup |f| \Bigl[  \int  \rho^{q_1\delta+1}_s \ d \gamma \Bigr]^{\frac{1}{q_1}} 
\\& \Bigl[ \frac{1}{t-s} \int\int_{s}^{t}   u^{q^*_1 q^*}\bigl( \delta (t-s) \mbox{div}_{\gamma} c\bigr)  \  \ dr\ d \gamma \Bigr]^{\frac{1}{q^*_1 q^*}}
\Bigl[ \frac{1}{t-s} \int\int_{s}^{t}   \rho^{q}_r  \ dr\ d \gamma \Bigr]^{\frac{1}{q^*_1 q}}.
\end{align*}
Applying again the H{\"o}lder inequality and Lemma \ref{AFlemma0} (see (\ref{lq-exp})) it is easy to show that the latter does not exceed $C |t-s|$ for some $C, \varepsilon$
where $t-s, \delta$ are chosen sufficiently small and $q^*_1$ close to $1$.

Analogously, we  estimate
\begin{align*}
&
 \int \rho^\delta_s(X_s) \bigl( f(X_t)  - f(X_s) \bigr) \rho_0 \ d \gamma  \le \| f\|_{\mbox{Lip}}  \int \rho^\delta_s(X_s)  \int_{s}^{t} \| c(X_r)  \|  \ dr \rho_0\ d \gamma
\\& 
\le (t-s) \| f\|_{\mbox{Lip}}   \Bigl[ \int \rho^{p_2 \delta+1}_s \ d \gamma \Bigr]^{\frac{1}{p_2}} \Bigl[ \int \frac{1}{t-s}\int_{s}^{t} \| c \|^{p^*_2}  \rho_r\ \ dr d \gamma\Bigr]^{\frac{1}{p^*_2}}
\\& \le
(t-s) \| f\|_{\mbox{Lip}}   \Bigl[ \int \rho^{p_2 \delta+1}_s \ d \gamma \Bigr]^{\frac{1}{p_2}} \Bigl[  \int  \| c\|^{p^*_2 p}  \  d \gamma\Bigr]^{\frac{1}{p^*_2 p}} 
\Bigl[ \frac{1}{t-s} \int \int_{s}^{t} \rho_r^{q}  \ \ dr d \gamma\Bigr]^{\frac{1}{p^*_2 q}}.
\end{align*}
Choosing $p^*_2$ close to $1$ and a sufficiently small $\delta$ we get the desired result.
\end{proof}

\begin{remark}
Below we generalize the existence result of \cite{AF} in infinite dimensions which has been established under the  assumption that
$\|c\|  \in L^p(\gamma)$, $p>1$. We prove it under the weaker assumption 
$c_i \in L^p(\gamma)$. Furthermore, we work with our slightly weaker notion of solution from Definition \ref{main-def} above.
\end{remark}

\begin{lemma}
\label{gauss-eq}
Let $\nu=\gamma = \prod_{i=1}^{\infty} \gamma_i$
be a product of the  standard Gaussian measures. Assume that
$
c = (c_i) : \R^{\infty} \to \R^{\infty}
$
is a mapping satisfying:
\begin{itemize}
\item[1)]
There exists $p>1$ such that
$
c_i  \in L^p(\gamma)
$
for every $i$.
\item[2)]
The divergence $\mbox{\rm div}_{\gamma} c$ satisfies
\begin{equation}
\label{FeynKac}
\exp\bigl( \varepsilon (\mbox{\rm div}_{\gamma} c)_{-} \bigr)
\in
L^1(\gamma).
\end{equation}
\end{itemize}

Then there exists $T=T(\varepsilon,p,q')>0$ such that the equation
$
\dot{\rho} + \mbox{\rm div}_{\gamma} \bigl( c \cdot \rho \bigr) =0
$
has a solution $\rho_t$ on $[0,T]$ for every initial condition
 $\rho_0 \in L^{q'}(\gamma)$ with some $q'>\frac{p}{p-1} =p^*$.
In addition, $\sup_{t \in [0,T]} \| \rho(t,\cdot)\|_{L^q(\gamma)}<\infty$.
\end{lemma}

\begin{proof}
Let us set:
$$
c_{(n)} = \sum_{i=1}^{n} \mathbb{E}^{\mathcal{F}_n}_{\gamma} c_i \cdot e_i
$$
and 
$$
 \rho_{0,n} \cdot\bigl(  \gamma  \circ P^{-1}_n \bigr)=
\bigl( \rho_0 \cdot \gamma \bigr)\circ P^{-1}_n.
$$
Equivalently, $\rho_{0,n} = \Exp^{\mathcal{F}_n}_{\gamma} \rho_0$.
Note that assumption 1) ensures that
$
c_{(n)}$ is well-defined.

It is well-known (and easy to check) that
$$
\mbox{div}_{\gamma} c_{(n)} = \mathbb{E}^{\mathcal{F}_n}_{\gamma} [\mbox{div}_{\gamma} c].
$$
This relation easily implies that
$[\mbox{div}_{\gamma} c_{(n)}]_{-} \le \mathbb{E}^{\mathcal{F}_n}_{\gamma} [\mbox{div}_{\gamma} c]_{-}$
and
$|\mbox{div}_{\gamma} c_{(n)}|^m \le \mathbb{E}^{\mathcal{F}_n}_{\gamma} |\mbox{div}_{\gamma} c|^m$, $m \ge 1$.

By convexity and  Jensen's
inequality one has 
$$
c_{(n)} \in L^p\bigl(\gamma \circ P^{-1}_n\bigr).
$$

Consider the equation
$$
\dot{\rho_n}  + \mbox{div}_{\gamma} (\rho_n \cdot c_{(n)} )=0
$$
with ${\rho_n}|_{t=0} = \rho_{0,n}$.
Since $\|\rho_{0,n} \|_{L^{q'}(\gamma)} \le \|\rho_{0} \|_{L^{q'}(\gamma)}$, we get by Lemma \ref{AFlemma} that there exists $T=T(\varepsilon,p,q')>0$ such that
this equation admits a solution $\rho_n$ on $[0,T]$
satisfying the following dimension-free bound
$$
M = \sup_{t \in [0, T], n \in \Nat} \| \rho_{n}(t, \cdot ) \|_{L^{q}(\gamma)} < \infty.
$$
For any function $\varphi \in \mathcal{C}$,  
the following identity holds:
\begin{equation}
\label{17.11.09}
\int \varphi \rho_n(t,x)\ d \gamma
=
\int \varphi \rho_n(0,x) \ d \gamma
+
\int_{0}^{t} \int \langle {c_{n}}, \nabla \varphi \rangle  \rho_n(s,x) \ d\gamma \ ds.
\end{equation}

Applying a diagonal argument one can extract a subsequence (which is denoted in what follows again by  $\rho_n$) such that
$
\{\rho_n(t,x) \}
$
converges weakly in $L^q(\gamma)$ to some function $\rho(t,x)$ for any $t$ from a dense countable subsequence $ I \subset [0,T]$.
Then $\{\rho_n(t,x) \}$ converges weakly in  $L^q(\gamma)$ for any $t \in [0,T]$ to a function denoted  in what follows by $\rho(t,x)$. Indeed, since 
$$\sup_n \| \rho_n(t,x)\|_{L^{q}(\gamma )}< \infty,$$ by a standard subsequence argument
it is enough to show that $\{\int f \rho_n(t,x)  \ d \gamma\}$ is a convergent sequence for every $f \in  L^p(\gamma)$. Clearly, it is sufficient to
check the claim for functions from $\mathcal{C}$. 
Since for such a function  $f \in \mathcal{C}$ the sequence $\{ \int f \rho_n(t,x) \ d \gamma\}$ is convergent for every $t \in I$, it  follows  easily from the estimate (\ref{Lp-Hoeld}) (we use here that $f$ is cylindrical, hence the right-hand side of
 (\ref{Lp-Hoeld}) depends on a finite collection of $c_i$) that
$\{ \int f \rho_n(t,x) \ d \gamma\}$ is a Cauchy sequence. Thus, we get that $\rho_n(t,x) \to \rho(t,x)$ weakly in $L^q(\gamma)$ for every $t \in [0,T]$.

One has for every smooth cylindrical function $\varphi= \varphi(x_1, \cdots, x_k)$
$$
\lim_n \Bigl(\int \varphi \rho_{n}(t,x)\ d \gamma
-
\int \varphi \rho_{n}(0,x) \ d \gamma
\Bigr) = 
\int \varphi \rho(t,x)\ d \gamma
-
\int \varphi \rho(0,x) \ d \gamma.
$$
Set: $g_n(s) = \int \langle {c_{n}}, \nabla \varphi \rangle  \rho_{n}(s,x) \ d\gamma $.
Using the convergence $c_{(n)} \to c_{}$
in $L^p(\gamma)$, one gets
$$
\lim_n g_n(s) =\lim_n  \int \langle {c_{n}}, \nabla \varphi \rangle  \rho_{n}(s,x) \ d\gamma 
= \int \langle {c}, \nabla \varphi \rangle  \rho(s,x) \ d\gamma  = g(s)
$$
for every $s \in [0,T]$. Clearly, $$
\sup_{s \in [0,T], n \in \Nat} |g_n(s)| \le  \sup_{s\in [0, T], n \in \Nat} \| \rho_{n}(s, \cdot ) \|_{L^{q}(\gamma)}
\| P_k \circ c\|_{L^{p}(\gamma)} \|\nabla \varphi\|_{L^{\infty}(\gamma)}.
$$
Then the Lebesgue dominated convergence theorem implies
$$
\int_{0}^{t} \int \langle {c_{n}}, \nabla \varphi \rangle  \rho_{n}(s,x) \ d\gamma \ ds
\to
\int_{0}^{t} \int \langle {c}, \nabla \varphi \rangle  \rho(s,x) \ d\nu \ ds.
$$ 
Passing to the limit in (\ref{17.11.09}) we get that $\rho$ is the desired solution.
\end{proof}

Before we proceed to the general case, let us explain
the main idea of the proof.
We construct a mapping $T$ pushing forward another measure $\mu$
onto   $\nu$. If $T$ is sufficiently smooth,
one can define the following new drift:
$$
c = DS(T) \cdot b(T),
$$
where $S$ is the inverse mapping to $T$.
One has
$$
\mu = \nu \circ S^{-1}
$$
and 
$$
(DT)^{-1} = DS(T).
$$
Let us give a heuristic proof of the key relation:
\begin{equation}
\label{div-map}
\mbox{div}_{\mu} c \circ S = \mbox{div}_{\nu} b .
\end{equation}
Take a test function $\varphi \in \mathcal{C}$.
One has
\begin{equation}
\label{bt-div}
\int \langle \nabla \varphi, c \rangle \circ S \ d \nu
=
\int \langle \nabla \varphi, c \rangle \ d \mu
= - \int \varphi  \ \mbox{div}_{\mu} c \ d \mu
= - \int \varphi(S) \ \mbox{div}_{\mu} c \circ S \ d \nu.
\end{equation}
On the other hand, we note that by the chain rule
\begin{equation}
\label{grad-chain}
 \nabla (\varphi(S))
 = (DS)^* \nabla \varphi(S).
\end{equation}
Hence
 $\int \langle \nabla \varphi, c \rangle \circ S\ d \nu$ is equal to
\begin{equation}
\label{bt-div2}
\int \langle (DS^*)^{-1} \nabla (\varphi(S)), c(S) \rangle \ d \nu
=
\int \langle \nabla (\varphi(S)), b \rangle \ d \nu
= - \int \varphi(S) \mbox{div}_{\nu} b \ d \nu.
\end{equation}
Obviously,  (\ref{bt-div}) and (\ref{bt-div2}) imply (\ref{div-map}).

Now let us try to  solve the equation
$$
\dot{\rho} + \mbox{div}_{\nu} (\rho \cdot b) =0
$$
for a wide class of probability measures.
Assume that $\nu$ is the image of
the standard  Gaussian measure
$\gamma$ under a mapping $T$.
Setting $c = DT^{-1} \cdot b(T) = DS(T) \cdot b(T)$
we transform the equation into 
\begin{equation}
\label{05.09.09}
\dot{g} + \mbox{div}_{\gamma} (g \cdot c) =0,
\end{equation}
where every    $\rho \cdot \mu$  is the image of $g \cdot \gamma$
 under $T$. Applying Lemma \ref{gauss-eq}
we obtain a solution to (\ref{05.09.09}). Then the function
$$
\rho(t,x) = g(t,T^{-1}(x))
$$
presents the desired solution. This follows immediately from the definition of solution in Definition \ref{main-def} and the change of variables formula.

\section{Sobolev estimates for triangular mappings}

Let $\nu$ be a probability measure on $\R^{\infty}$. 

{\bf Assumption.} Throughout the paper it is assumed  that
 for every $i \in \mathbb{N}$ there exists a function $\beta_i \in L^1(\nu)$ such that 
$$
\int \partial_{e_i} \varphi \ d\nu = - \int  \varphi \beta_i \ d\nu .
$$
for every $\varphi \in \mathcal{C}$. The function
$\beta_i$ is called logarithmic derivative of $\nu$ along $e_i$

\begin{remark}
{ This assumption implies the following important property:  all the projections
$$
\nu_n = \nu \circ P^{-1}_n,
$$
where
$
P_n(x) = (x_1,\cdots, x_n)
$,
have Lebesgue densities.
 This follows, for instance, from Lemma 2.1.1 of \cite{Nualart}. According to this Lemma, every measure $\mu$ on $\mathbb{R}^d$ which satisfies inequality
$$
\Bigl| \int \partial_{x_i} \varphi d \mu \Bigr| \le C \sup_x|\varphi(x)|, \ \varphi \in C^{\infty}_0(\mathbb{R}^d)
$$
for some $C$ independend on $\varphi$, is absolutely continuous.
}
\end{remark}

\begin{remark}
It is important to keep in mind that { also} the projections $\nu_n$  have logarithmic derivatives given by the conditional expectations $\mathbb{E}^{\mathcal{F}_n}_{\nu} \beta_i$.
\end{remark}

Consider another  Borel probability measure  $\mu$ on $\R^{\infty}$.
We denote by $\mu_{i}=\mu \circ P^{-1}_{i}$ the projection of $\mu$
onto the subspace generated by the first $i$ basis vectors. Recall that throughout the paper $\mu_i$ is assumed to have 
a Lebesgue density, which will be  denoted by $\rho_{\mu_{i}}$.
For every fixed
$x=(x_1, \cdots, x_{i-1})$ we denote by
$\mu^{\bot}_{x,i}$
 the corresponding {\bf one-dimensional} conditional measure
obtained from the disintegration
of $\mu_{i}$ with respect to
$\mu_{i-1}$.
Note that $\mu_{i-1} =\mu_{i} \circ P^{-1}_{i-1}$.
These measures are related by the following identity
\begin{align*}
\int \varphi \rho_{\mu_i} dx & = \int \varphi(x,x_i) \ \mu_{i}(dx dx_i)
=  \int  \Bigl( \int \varphi(x,x_i)  \mu^{\bot}_{x,i}(dx_i) \Bigr) \  \mu_{i-1}(dx)
\end{align*}
for all bounded Borel $\varphi : \mathbb{R}^i \to \mathbb{R}$.

If for $\mu_i$-almost points $x$ the corresponding conditional measures $\mu^{\bot}_{x,i}$ have Lebesgue densities, they will be denoted by  
$\rho_{\mu^{\bot}_{x,i}}$. In this case the latter formula reads as
$$
\int \varphi \rho_{\mu_i} dx =
 \int  \Bigl( \int \varphi(x,x_i) \rho_{\mu^{\bot}_{x,i}} d  x_i \Bigr) \ \rho_{\mu_{i-1}}(x) \ dx.
$$

In this section we study a-priori estimates for so-called triangular  mappings, which are also known as "Knothe mappings''.
We call a mapping $T: \R^{\infty} \to \R^{\infty}$ triangular if
it has the form
$$
T = \sum_{i=1}^{\infty} T_i(x_1, \cdots, x_i) e_i
$$
and, in addition, $x_i \to T_i(x_1, \cdots, x_i)$ is an increasing function.

Given two probability measures $\mu$ and $\nu$ on $\mathbb{R}^{\infty}$ we are looking for a triangular mapping $T : \mathbb{R}^{\infty} \to \mathbb{R}^{\infty}$
pushing forward $\mu$ onto $\nu$.
The proof of  existence of mappings of such   type  
on $\mathbb{R}^{\infty}$  for a broad class of measures  can be found in (\cite{BKM}, \cite{B2006}).
It relies on the fact that $T$ can be precisely described in terms of conditional probabilities of $\mu$ and $\nu$.  
In the one-dimensional case $T = T_{\mu,\nu}$ is defined by the relation
$$
\int_{-\infty}^{x} \rho_{\mu}(t) \ dt = \int_{-\infty}^{T(x)} \rho_{\nu}(t) \ dt.
$$
In the finite- and infinite-dimensional case $T$ is obtained by induction
\begin{itemize}
\item[1)]
$T_1$ is the increasing transport of the projections on the first coordinate
\begin{equation}
\label{t1}
T_1(x_1) = T_{\mu_1,\nu_1}(x_1)
\end{equation}
\item[2)]
$T_i$, $i >1$, is the increasing transport of the one-dimensional conditional measures $\tilde{\mu}$, $\tilde{\nu}$:
\begin{equation}
\label{ti}
T_i (x_1, \cdots,x_i) = T_{\tilde{\mu}, \tilde{\nu}}(x_i),
\end{equation}
where
$\tilde{\mu} = \mu^{\bot}_{x,i}$, $\tilde{\nu} = \nu^{\bot}_{T_{i-1}(x),i}$, $x=(x_1, \cdots, x_{i-1})$.
\end{itemize}

\begin{remark}
{ One of the sources of difficulty in our approach is that the general infinite-dimensional change of variables does not preserve membership in $\mathcal{C}$. We don't have this problem
in the case  of triangular change of variables.
}
\end{remark}

The existence result and the basic properties formulated in the following theorem have been proved in  papers \cite{BoKo05(2)}, \cite{BKM}.
\begin{theorem}
\label{tri}
Let $\mu$ be a probability measure on $\mathbb{R}^{\infty}$  satisfying the following assumptions:
\begin{itemize}
\item[1)] Any projection $\mu_i$, $i \in \mathbb{N}$, is  absolutely continuous measure with respect  to Lebesgue measure on $\mathbb{R}^i$.
\item[2)] For $\mu_i$-almost all $x$ the corresponding conditional measures $\mu^{\bot}_{x,i}$ are absolutely continuous, with respect  to Lebesgue measure on $\mathbb{R}$.
\end{itemize}
Then, for all probability measures $\nu$ in $\mathbb{R}^{\infty}$  there exists a  triangular mapping $T$  pushing forward $\mu$ onto $\nu$.
The mapping $T$ is unique up to a set of $\mu$-measure  zero.

In addition, if $\nu$ also satisfies (1) and (2), then  there exists a  triangular mapping $S$  pushing forward $\nu$ onto $\mu$.
In addition, they are reciprocal:
$$
T \circ S = \mbox{\rm{Id}} \ \ \ \ \  \nu-{\rm{a.e.}}
$$
$$
S \circ T = \mbox{\rm{Id}} \ \ \ \ \  \mu-{\rm{a.e.}}
$$
\end{theorem}

\begin{remark}
{ Since every $\mu_i$ is absolutely continuous with respect  to Lebesgue measure on $\mathbb{R}^i$, we get immediately that the corresponding 
conditional measures $\mu^{\bot}_{x,i}$ are absolutely continuous for $\mu_{i-1}$-almost all $x$. Moreover, $\mu^{\bot}_{x,i} = \rho_{\mu_{x,i}^{\bot}} dx_i$ admit logarithmic derivatives which are related to logarithmic
derivatives of $\mu$ in the following way:
\begin{equation}
\label{exp-partder}
\Exp^{\mathcal{F}_i}_{\mu} \beta_i
=
\frac{  \partial_{x_i}\rho_{\mu_{x,i}^{\bot}}}{\rho_{\mu_{x,i}^{\bot}}}
\end{equation}
(the right-hand side is assumed to be zero if $\rho_{\mu_{x,i}^{\bot}}=0$).
Indeed, to see this let us take  smooth functions $\varphi(x_1, \cdots, x_i), \xi(x_1, \cdots, x_{i-1})$. One has
\begin{align*}
  - \int_{\mathbb{R}^{d-1}} & \xi \Bigl(  \int_{\mathbb{R}} \varphi  \ \Exp^{\mathcal{F}_i}_{\mu}  \beta_i \ d \mu^{\bot}_{x,i}  \Bigr)  d \mu_{i-1}(x)  = - \int \xi \varphi \Exp^{\mathcal{F}_i}_{\mu} \beta_i  d \mu  
\\& = - \int \xi \varphi \beta_i  d \mu  = \int \xi \varphi_{x_i} d \mu = \int_{\mathbb{R}^{i-1}} \xi \Bigl( \int_{\mathbb{R}} \varphi_{x_i} d \mu^{\bot}_{x,i} \Bigr)  \mu_{i-1}(dx).
\end{align*}
This relation immediately  implies that $\int_{\mathbb{R}} \varphi  \ \Exp^{\mathcal{F}_i}_{\mu}  \beta_i \ d \mu^{\bot}_{x,i} 
= \int_{\mathbb{R}} \varphi_{x_i} d \mu^{\bot}_{x,i}$ and we get the claim.

 In particular, the measures satisfying  our general assumptions from Section 2 do satisfy the assumptions of  Theorem \ref{tri}.}
\end{remark}

Note that in the one-dimensional case $T$ and $S$ are just non-decreasing mappings which can be written exactly in terms of the distribution functions of $\mu$ and $\nu$.
Hence $T$ {(resp. $S$)} admits classical pointwise derivative $T'$ $\mu$ {(resp. $\nu$)}-almost everywhere. One can easily check that $T'$ is $\mu$-a.e. positive, because
otherwise $\nu$ has a non-trivial singular part. In particular
$$
T'(S) \cdot S' = 1, \ \ \ \mu-{\rm{a.e.}}
$$

\begin{remark}
Since every $T_i$ is constructed as a one-dimensional increasing transportation of conditional measures, the following generalization of the above relation
$$
\partial_{x_i} T_i (S) \cdot \partial_{x_i} S_i =1 \ \ \ \mu-{\rm{a.e.}}.
$$
is valid in the finite- and infinite-dimensional case.
\end{remark}

If $T$ and $S$ are smooth (meaning that every function $T_i$, $S_i$ is smooth) then their Jacobian matrices are triangular:
$$
DT  = \left( \begin{array}{cccccccc}
\partial_{x_1} T_1  & 0                              & 0 &  \cdots & 0  & 0 & 0          & \cdots \\
\partial_{x_1} T_2  & \partial_{x_2} T_2  & 0 & \cdots   & 0 & 0 & 0          &  \cdots\\
\cdots  & \cdots  & \cdots & \cdots & \cdots  & \cdots & \cdots  & \cdots \\
\partial_{x_1} T_i  & \partial_{x_2} T_i  &  \partial_{x_3} T_i & \cdots  &  \partial_{x_{i-1}} T_i &  \partial_{x_i} T_i  & 0 & \cdots \\
\cdots  & \cdots  & \cdots & \cdots & \cdots  & \cdots & \cdots & \cdots 
\end{array}
\right)
$$
$$
DS  = \left( \begin{array}{cccccccc}
\partial_{x_1} S_1  & 0                              & 0 &  \cdots & 0  & 0 & 0          & \cdots \\
\partial_{x_1} S_2  & \partial_{x_2} S_2  & 0 & \cdots   & 0 & 0 & 0          &  \cdots\\
\cdots  & \cdots  & \cdots & \cdots & \cdots  & \cdots & \cdots  & \cdots \\
\partial_{x_1} S_i  & \partial_{x_2} S_i  &  \partial_{x_3} S_i & \cdots  &  \partial_{x_{i-1}} S_i &  \partial_{x_i} S_i  & 0 & \cdots \\
\cdots  & \cdots  & \cdots & \cdots & \cdots  & \cdots & \cdots & \cdots 
\end{array}
\right)
$$
In addition, if all $\partial_{x_i} T_i(x) \ne 0$ at $x$ (which happens $\mu$-a.e.), then
$$
 DS(T(x)) = (D T)^{-1}.
$$

In this section we establish global Sobolev estimates for  triangular mappings.
 Note that some (dimension-dependent) Sobolev estimates for the triangular mappings
 have been obtained  in \cite{ZhOv}. See also   \cite{Kol2010} for similar results on
optimal transportation. 

\begin{definition}
Let $\nu$ be a probability measure { with logarithmic derivative $\beta_i$} on $\mathbb{R}^{\infty}$ and $f \in L^1(\nu)$.
We say that $\partial_{x_i} f$ is a Sobolev partial derivative of $f$
if $f \beta_i \in L^1(\nu)$ and $$
\int \partial_{x_i} f \cdot \varphi \ d\nu = - \int f  \cdot \partial_{x_i}  \varphi \ d\nu - \int  f \varphi \beta_i \ d \nu
$$
for every $\varphi \in \mathcal{C}$.
Obviously, this defines $\partial_{x_i} f $ uniquely.
\end{definition}

\begin{definition}
Let $p \in (1,\infty)$ and $\nu$ be a probability measure such that $\beta_i \in L^{p^*}(\nu)$ for all $i$.
We say that a function $f$ belongs to the Sobolev class $W^{1,p}(\nu)$ if 
$$
\| f\|_{L^p(\nu)} + \Bigl\| \bigr( \sum_{i=1}^{\infty} f^2_{x_i} \bigr)^{1/2} \Bigr\|_{L^p(\nu)} < \infty.
$$
If, in addition, every $f_{x_i}$ has all Sobolev partial derivatives in $L^p(\nu)$ and 
$$
\| D^2 f \|_{\mathcal{HS}} = \bigl( \sum_{i,j} f^2_{x_i x_j} \bigr)^{\frac{1}{2}} \in L^p(\nu),
$$
we say that $f \in W^{2,p}(\nu)$.
\end{definition}

\begin{remark}
Though we shall not use this below, it follows by \cite{AR90} that due to the assumption
that $\beta_i \in L^{p^*}(\nu)$ for all $i \in \mathbb{N}$ both $W^{1,p}(\nu)$ and
$W^{2,p}(\nu)$ are complete.
\end{remark}

All  derivatives of $S$ and $T$ below will be understood in the Sobolev sense, with respect to $\nu$ and $\mu$ respectively.

We start with a one-dimensional estimate.

\begin{proposition}
\label{1d-pow-trans}
Let $\mu = e^{-V} dx$, $\nu = e^{-W} dx$ be two probability measures on $\R$ with continuously differentiable functions
$V$ and $W$.  Consider the   increasing mapping $T$ pushing forward $\mu$
onto $\nu$. Assume that
\begin{equation}
\label{1d-pow-trans0}
 x W'(x) \ge \frac{ -1 + \varepsilon}{p-1}
\end{equation}
for some $p>1$ and $\varepsilon>0$.
Assume, in addition, that $|x|^{\frac{p(p+\delta)}{\delta}} \in L^1(\nu) , \ |V'|^{p+\delta} \in L^1(\mu)$ for some $\delta>0$.
Then $T' :=DT \in L^p(\mu)$.
\end{proposition}
\begin{proof}
By the change of variables formula $\int_{-\infty}^{x} e^{-V(t)} \ dt = \int_{-\infty}^{T(x)} e^{-W(t)} \ dt$.  Clearly, this implies that $T$ is differentiable and 
$$
e^{-V} = T' e^{-W(T)}.
$$
Moreover, $T' = e^{W(T) - V}$ is continuously differentiable and satisfies
$$
T^{''} = T' \bigl( W'(T) T' - V'\bigr).
$$
Take a positive test function $\xi$. Integrating by parts one obtains
\begin{align*}
& \int_{\mathbb{R}} (T')^p  e^{-V} \xi \ dx =
 \\ &=
-(p-1) \int_{\mathbb{R}}  T (T')^{p-2} T^{''}  e^{-V} \xi \ dx
+
\int_{\mathbb{R}} T (T')^{p-1} V'  e^{-V} \xi \ dx
- \int_{\mathbb{R}} (T')^{p-1} e^{-V} \xi' \ dx
\\&
=
-(p-1) \int_{\mathbb{R}}  T W'(T) (T')^{p}   e^{-V} \xi \ dx
+
p \int_{\mathbb{R}} T (T')^{p-1} V'  e^{-V} \xi \ dx -
\int_{\mathbb{R}} (T')^{p-1} e^{-V} \xi' \ dx.
\end{align*}
{One obtains
$$
p\int_{\mathbb{R}} T (T')^{p-1} V'  e^{-V} \xi \ dx
- \int_{\mathbb{R}} (T')^{p-1} e^{-V} \xi' \ dx
= \int_{\mathbb{R}} (T')^p  e^{-V} \xi \ dx 
+ 
(p-1) \int_{\mathbb{R}}  T W'(T) (T')^{p}   e^{-V} \xi \ dx.
$$
Then assumption (\ref{1d-pow-trans0}) implies}
\begin{align*}
\varepsilon \int_{\mathbb{R}} (T')^p & e^{-V} \xi \ dx \le
p \int_{\mathbb{R}} T (T')^{p-1} V'  e^{-V} \xi \ dx -
 \int_{\mathbb{R}} (T')^{p-1} e^{-V} \xi' \ dx
 \\& \le
\frac{\varepsilon}{2} \int_{\mathbb{R}} (T')^p  e^{-V} \xi \ dx
+
N(\varepsilon,p)
 \int_{\mathbb{R}} T^p (V')^p  e^{-V} \xi \ dx
 -
 \int_{\mathbb{R}} (T')^{p-1} e^{-V} \xi' \ dx.
\end{align*}
By the H{\"o}lder inequality
\begin{align*}
\int_{\mathbb{R}} & T^p (V')^p  e^{-V} \xi \ dx
\le
\int_{\mathbb{R}} (V')^{p+\delta} e^{-V} \xi \ dx + C(\delta,p) \int_{{\mathbb{R}}} T^{\frac{p(p+\delta)}{\delta}}   e^{-V} \xi \ dx
\\&
= \int_{\mathbb{R}} (V')^{p+\delta} e^{-V} \xi \ dx + C(\delta,p)  \int_{R} |x|^{\frac{p(p+\delta)}{\delta}}   e^{-W} \xi(T^{-1}) \ dx,
\end{align*}
and
$$
 - \int_{\mathbb{R}} (T')^{p-1} e^{-V} \xi' \ dx
\le \frac{\varepsilon}{4}  \int_{{\mathbb{R}}} (T')^p e^{-V} \xi \ dx + c(p,\varepsilon) \int_{{\mathbb{R}}} \Bigl|\frac{\xi'}{\xi} \Bigr|^{p} \xi  e^{-V} \ dx.
$$
Thus, we obtain a bound for   $\int_{\mathbb{R}} (T')^p (V')^p  e^{-V} \xi \ dx$ and $ - \int_{\mathbb{R}} (T')^{p-1} e^{-V} \xi' \ dx$.
Taking a suitable sequence $\{ \xi_n\}$ with $\xi_n \to 1$, $\xi_n \le 1$,  and $\lim_{n} \int_{{\mathbb{R}}} \Bigl|\frac{\xi'_n}{\xi_n} \Bigr|^{p} \xi_n  e^{-V} \ dx =0$ we complete the proof.
\end{proof}

Now let us come back to the infinite-dimensional case.
Below in the proofs we apply the following scheme.
\begin{itemize}
\item[1)]
Prove the statement for smooth positive densities.
\item[2)]
Approximate the Sobolev densities by smooth positive densities and deduce the desired estimates. 
\end{itemize} 

First, we need an approximation lemma.

{
\begin{remark}
Everywhere below we agree that the functions $\beta_i = \partial_{x_i} \rho$ vanish on the set ${\{\rho =0\}}$.
In particular, $ \int  \Bigl\|  \frac{\nabla \rho}{\rho}\Bigr\|^p \rho \ dx =  \int_{\rho>0}  \Bigl\|  \frac{\nabla \rho}{\rho}\Bigr\|^p \rho \ dx$.

The same agreement concerns the logarithmic derivatives of higher order
$$\partial_{x_j} \beta_i  = \frac{\partial_{x_i x_j} \rho}{\rho}  -  \frac{\partial_{x_i}\rho\cdot  \partial_{x_j} \rho}{\rho^2}.$$
\end{remark}

\begin{lemma}\label{approx}
For every probability measure $\nu = \rho \ dx$ on $\mathbb{R}^d$ satisfying $\beta_i = \frac{\rho_{x_i}}{\rho} \in L^p(\nu)$  
for some $p \ge 1$  (with $\beta_i:=0$ on $\{\rho=0\}$)  and all $1 \le i \le d$
there exists a sequence of probability measures $\nu_n = \rho_n \ dx$  such that 
\begin{itemize}
\item[1)] $\nu_n \to \nu$ in variation norm,
\item[2)] $\lim_{n} \int  \Bigl\|  \frac{\nabla \rho_n}{\rho_n} \Bigr\|^p \rho_n \ dx = \int  \Bigl\|  \frac{\nabla \rho}{\rho}\Bigr\|^p \rho \ dx$.
\item[3)] every $\rho_n$ is smooth and nonnegative,
\item[4)] the partial derivatives ${\partial_{x_i} \rho_n}$ are uniformly bounded and integrable for every $i, n$.
\end{itemize}
Moreover, if every logarithmic derivative $\beta_{i}$ has Sobolev derivative along any coordinate $x_j$ and if, in addition, there exists $p \ge 1$ such that 
$$\partial_{x_j} \beta_i  = \frac{\partial_{x_i x_j} \rho}{\rho}  -  \frac{\partial_{x_i}\rho\cdot  \partial_{x_j} \rho}{\rho^2} \in L^{p}(\nu), \ \ \ \beta_{i} \in L^{2p}(\nu)$$ then
item 2) can be strengthened as follows:   $$\lim_{n} \int  \Bigl\|  \frac{\nabla \rho_n}{\rho_n} \Bigr\|^{2p} \rho_n \ dx = \int  \Bigl\|  \frac{\nabla \rho}{\rho}\Bigr\|^{2p} \rho \ dx,$$  $$\lim_n \int \bigl|\frac{\partial_{x_i x_j} \rho_n}{\rho_n} \bigr|^p \rho_n \ dx = 
\int \Bigl|\frac{\partial_{x_i x_j} \rho}{\rho} \Bigr|^p \rho \ dx.$$
\end{lemma}
{\bf Sketch of the proof:}
The arguments are quite standard (compare with the proof of the classical result that $C^{\infty}_0$-functions are dense in Sobolev spaces, Theorem 2.1.7 in \cite{B2008}) and we only give a sketch here.
Without loss of generality we assume that $\rho$ is compactly supported. Otherwise  one can approximate $\rho$ by $ \varphi_n \cdot \rho$, where $\{\varphi_n\}$ is a sequence of smooth  compactly supported functions $0 \le \varphi_n \le 1$ such that
$\varphi_n \to 1$  and $\partial_{x_i} \varphi_n \to 0$ pointwise.
In addition, we assume that 
\begin{equation}
\label{supremgrad}\sup_{x}  \Bigl\|  \frac{\nabla \varphi_n}{\varphi_n} \Bigr\|^p \varphi_n <\infty
\end{equation}
(again, we agree that $\frac{\nabla \varphi_n}{\varphi_n}=0$ if $\varphi_n(x)=0$).
Functions $\varphi_{n,1} : \mathbb{R} \mapsto [0,1]$ of this type are easy to construct in the one-dimensional case (for all values of $p$ simultaneously!) and in the multy-dimensional case it is sufficient to take a 
product of their independend copies: $\varphi_n = \prod_{i=1}^d \varphi_{n,1}(x_i)$.

We apply the following inequality which  can be easily checked by elementary means:
$$|a+b|^p - |a|^p \le c_p \bigl( |a|^{p-1} |b| + |b|^p\bigr), \ p \ge 1.$$
Then it follows from the H{\"o}lder inequality
\begin{align}
\label{16.12.13}
\int \Bigl| \   \Bigl\| \frac{\nabla \rho}{\rho} & + \frac{\nabla \varphi_n}{\varphi_n} \Bigr\|^p \rho \varphi_ n  - \Bigl\| \frac{\nabla \rho}{\rho} \Bigr\|^p \rho \varphi_ n  \Bigr| \ dx
\nonumber
\\& \le c_p \Bigl[ \Bigl(\int  \Bigl\|  \frac{\nabla \varphi_n}{\varphi_n} \Bigr\|^p \rho \varphi_ n \ dx \Bigr)^{\frac{1}{p}} \Bigl( \int \Bigl\| \frac{\nabla \rho}{\rho} \Bigr\|^p \varphi_n \rho  \ dx \Bigr)^{\frac{p-1}{p}}
+ \int  \Bigl\|  \frac{\nabla \varphi_n}{\varphi_n} \Bigr\|^p \rho \varphi_ n \ dx
\Bigr]. 
\end{align}
It follows from (\ref{supremgrad}) and pointwise convergence $\partial_{x_i} \varphi_n \to 0$ that $\lim_n \int \bigl\| \frac{\nabla \varphi_n}{\varphi_n} \bigr\|^p \rho \varphi_ n dx =0$ by the Lebesgue convergence theorem. In addition, 
$\lim_n \int \Bigl\| \frac{\nabla \rho}{\rho} \Bigr\|^p \rho \varphi_ n dx = \int \Bigl\| \frac{\nabla \rho}{\rho} \Bigr\|^p \rho dx$. Clearly, we get from (\ref{16.12.13}) $$\lim_{n} \int  \Bigl\| \frac{\nabla \rho}{\rho} + \frac{\nabla \varphi_n}{\varphi_n} \Bigr\|^p \rho \varphi_ n  \ dx = \int  \Bigl\|  \frac{\nabla \rho}{\rho}\Bigr\|^p \rho \ dx$$ 
and approximation by compactly supported functions is justified.

As soon as we deal with compactly supported density $\rho$, we set $\rho_n = P_{\frac{1}{n}} \rho$, where $P_t$ is  the standard heat semigroup, which is a convolution with 
smooth kernels $(2 \pi t)^{\frac{d}{2}} e^{-\frac{x^2}{2 t^2}}$. Then the classical arguments give us, in particular, that $\rho_n \to \rho$, $\partial_{x_i} \rho_n  \to \partial_{x_i} \rho$ almost everywhere and
$ \int  \Bigl\|  \frac{\nabla \rho_n}{\rho_n} \Bigr\|^p \rho_n \ dx \le  \int  \Bigl\|  \frac{\nabla \rho}{\rho}\Bigr\|^p \rho \ dx$. 
Indeed, let us show the latter inequality. One has
$$
\|\nabla \rho_n\|  = \|\nabla P_{\frac{1}{n}} \rho\| = \|P_{\frac{1}{n}} \nabla \rho\| \le  \Bigl( P_{\frac{1}{n}} \Bigl( \frac{|\nabla \rho|^p}{\rho^{p-1}} \Bigr)\Bigr)^{\frac{1}{p}} \Bigl( P_{\frac{1}{n}} \rho \Bigr)^{\frac{1}{q}}
=   P_{\frac{1}{n}} \Bigl( \frac{|\nabla \rho|^p}{\rho^{p-1}} \Bigr)^{\frac{1}{p}} \rho_n^{\frac{1}{q}}
$$
Hence
$
 \Bigl( \frac{\|\nabla \rho_n\|}{\rho_n} \Bigr)^p \rho_n  \le  P_{\frac{1}{n}} \Bigl( \frac{\|\nabla \rho\|^p}{\rho^{p-1}} \Bigr).
$ Integrating over $\mathbb{R}^d$ and using the semigroup property $\int P_t f dx = \int f dx$ one gets the desired estimate.
Passing to the limit and applying the Fatou lemma  one gets 2). The other properties are  easy to check.

The convergence result for integrals over higher-order derivatives is obtained in the same way.  In the proof a similar  assumption of the type (\ref{supremgrad}) is required:
$\label{sup-grad}\sup_{x}  \Bigl\|  \frac{\partial_{x_i x_j} \varphi_n}{\varphi_n} \Bigr\|^p \varphi_n <\infty$ for all $i,j$.

Let us also stress that finiteness of $\int \Bigl|\frac{\partial_{x_i x_j} \rho}{\rho} \Bigr|^p \rho \ dx$  implies naturally 
the finiteness of $\int  \Bigl\|  \frac{\nabla \rho}{\rho} \Bigr\|^{2p} \rho \ dx$. Indeed, making a formal integration by parts, we obtain
\begin{align*}
\int    \frac{(\partial_{x_i} \rho)^{2p}}{\rho^{2p}}  \rho \ dx & = \frac{1}{-2p+2} \int (\rho^{-2p+2})_{x_i} \rho^{2p-1}_{x_i} dx
 = \frac{2p-1}{2p-2} \int \rho^{-2p+2} \rho^{2p-2}_{x_i} \rho_{x_i x_i} dx 
\\& =  \frac{2p-1}{2p-2} \int  \frac{\rho^{2p-2}_{x_i}}{\rho^{2p-2}} \frac{\rho_{x_i x_i}}{\rho} \rho dx
\le \varepsilon \int    \frac{(\partial_{x_i} \rho)^{2p}}{\rho^{2p}}  \rho \ dx + C_{p,\varepsilon} \int \Bigl|\frac{\partial_{x_i x_i} \rho}{\rho} \Bigr|^p \rho \ dx.
\end{align*}
Thus $\int    \frac{(\partial_{x_i} \rho)^{2p}}{\rho^{2p}}  \rho \ dx \le c(p)  \int \Bigl|\frac{\partial_{x_i x_i} \rho}{\rho} \Bigr|^p \rho \ dx$.
}

\begin{remark}
\label{tri-smooth}
It is straightforward to check using  (\ref{t1}), (\ref{ti}) that $T$ and $S$ are continuously differentiable, $\mu = e^{-V} \ dx$, $\nu = e^{-W} \ dx$, and $V, W$  have uniformly bounded derivatives.
Note that in this case all conditional measures have positive densities and all the derivatives $\partial_{x_i} S_i, \partial_{x_i} T_i$ are positive. More precise statements about the 
regularity of triangular mappings can be found in 
 \cite{BKM} (Lemma 2.6) and \cite{ZhOv}.
\end{remark}

\begin{proposition}
\label{l2triest}
Consider the  triangular mapping $S$ pushing forward   $\nu$ onto $\gamma$.
Assume that $\beta_i \in L^2(\nu)$ for all $i$.
Then for every $i$  the mapping $S_i$ belongs to $W^{1,2}(\nu)$.
In particular, the following estimates hold:
$$
\int \bigl( \partial_{x_i} S_i \bigr)^2  d\nu
\le
  \int  \bigl(\Exp^{\mathcal{F}_i}_{\nu} \beta_i\bigr)^2   d\nu,
$$
$$
\int \bigl( \partial_{x_j} S_i \bigr)^2  d\nu
\le
\int \Bigl( \Exp^{\mathcal F_{i}}_{\nu} \beta_j  \Bigr)^2   d\nu
-  \int \Bigl( \Exp^{\mathcal F_{i-1}}_{\nu} \beta_j  \Bigr)^2  \ d\nu, \ i > j.
$$
In particular,
$$
\|\partial_{x_j}  S \|^2 = \sum_{i \ge j}
\int \bigl( \partial_{x_j} S_i \bigr)^2  d\nu \le
  \int   \beta_j^2   d\nu.
$$
\end{proposition}
\begin{proof}
First, we note that due to the finite-dimensional structure of triangular mappings it is sufficient to establish the statement
for finite-dimensional measures. We start with the case when
$\rho_{\nu} = e^{-V}$, where $V$ is a smooth function on $\mathbb{R}^d$ with uniformly bounded derivatives.

In the proof we apply the following relation between the logarithmic derivatives and conditional densities of the corresponding projections (see (\ref{exp-partder}))
$$
\Exp^{\mathcal{F}_i}_{\nu} \beta_i
=
\frac{  \partial_{x_i}\rho_{\nu_{x,i}^{\bot}}}{\rho_{\nu_{x,i}^{\bot}}}. 
$$
We keep the notation $\rho_{\gamma}$ for the Lebesgue density of the $1$-dimensional standard Gaussian measure $\gamma$:
$$
\rho_{\gamma} = \frac{1}{\sqrt{2 \pi}} e^{-\frac{t^2}{2}}.
$$
According to Remark \ref{tri-smooth} all the functions $S_i$
are continuously differentiable and $\partial_{x_i} S_i >0$.
It follows  by the change of variables formula  that
$
\int S^2_i \ d \nu = \int x_i^2 \ d \gamma,
$
hence $S_i \in L^2(\nu)$. This implies that
$\partial_{x_i} S_i \in L^1(\nu)$. Indeed
$$
\int \partial_{x_i} S_i \ d \nu =  - \int  S_i \beta_i \ d \nu \le \|S_i\|_{L^2(\nu)} \|\beta_i\|_{L^2(\nu)}.
$$
Let us estimate $\int \bigl( \partial_{x_i} S_i \bigr)^2 \ d\nu$.
One has the following explicit formula for $S_i$ (we stress that the expression below makes sense because $\rho_{\nu}$ is positive as well as the densities of 
its projections and conditional measures):
\begin{equation}
\label{Si-nu}
\int_{-\infty}^{S_i(x,x_i)} \rho_{\gamma}(t) \ dt = \int_{-\infty}^{x_i} \rho_{\nu^{\bot}_{x,i}}(t) \ d t,
\end{equation}
where
$$
\rho_{\nu^{\bot}_{x,i}}(t)
=
\frac{\rho_{\nu_i}(x,t)}{\int_{-\infty}^{\infty} \rho_{\nu_i}(x,t) \ dt}.
$$
Differentiating (\ref{Si-nu}) along $x_i$ one obtains
\begin{equation}
\label{onedim-cv}
\frac{\rho_{\nu_{x,i}^{\bot}}}{\rho_{\gamma}(S_i)}
=
\partial_{x_i} S_i.
\end{equation}
Formally applying integration by parts we get
\begin{align*}
\int \bigl( \partial_{x_i} S_i \bigr)^2 \ d\nu
= &
\int  \partial_{x_i} S_i \frac{\rho_{\nu_{x,i}^{\bot}}}{\rho_{\gamma}(S_i)}  \ d\nu
= - \int  \partial_{x_i} S_i \cdot S^2_i \frac{\rho_{\nu_{x,i}^{\bot}}}{\rho_{\gamma}(S_i)}  \ d\nu
\\&
- 2 \int  S_i  \frac{ \partial_{x_i}\rho_{\nu_{x,i}^{\bot}}}{\rho_{\gamma}(S_i)}  \ d\nu
\le
\int \frac{1}{\partial_{x_i} S_i} \Bigl(\frac{  \partial_{x_i}\rho_{\nu_{x,i}^{\bot}}}{\rho_{\nu_{x,i}^{\bot}}} \Bigr)^2
\frac{\rho_{\nu_{x,i}^{\bot}}}{\rho_{\gamma}(S_i)}  \ d\nu =
\\&
=
\int  \Bigl(\frac{  \partial_{x_i}\rho_{\nu_{x,i}^{\bot}}}{\rho_{\nu_{x,i}^{\bot}}} \Bigr)^2
  \ d\nu
  =
  \int  \bigl(\Exp^{\mathcal{F}_i}_{\nu} \beta_i\bigr)^2
  \ d\nu.
\end{align*}
To justify the above computation 
we integrate not over $\nu$ but over $\xi \cdot \nu$, where $\xi$ is a compactly supported smooth function 
on $\mathbb{R}^d$.
 By the same arguments one gets
$$
\int \bigl( \partial_{x_i} S_i \bigr)^2 \xi \ d\nu
\le
(1+\varepsilon) \int  \bigl(\Exp^{\mathcal{F}_i}_{\nu} \beta_i\bigr)^2 \xi \ d \nu
+ c(\varepsilon)  \int \bigl( \frac{\partial_{x_i} \xi}{\xi}\bigr)^2 \xi \ d \nu.
$$
Choosing  ''an appropriate'' convergent sequence $\xi_k \to 1$ with $\lim_k \int \bigl( \frac{\partial_{x_i} \xi_k}{\xi_k}\bigr)^2 \xi_k \ d \nu = 0$ one easily gets the desired result.

Analogously, one has for $\partial_{x_j} S_i$, $i \ne j$:
$$
\rho_{\gamma}(S_i) \partial_{x_j} S_i =
\frac{\int_{-\infty}^{x_i} \partial_{x_j} \rho_{\nu_i}(x,t) \ dt}{\int_{-\infty}^{\infty} \rho_{\nu_i}(x,t) \ dt}
-
\frac{\int_{-\infty}^{x_i}  \rho_{\nu_i}(x,t) dt \
\int_{-\infty}^{\infty} \partial_{x_j} \rho_{\nu_i}(x,t) \ dt \cdot }{\Bigl(\int_{-\infty}^{\infty} \rho_{\nu_i}(x,t) \ dt
\Bigr)^2}.
$$

Denoting the right-hand side by $f$ one gets
\begin{equation}
\label{i-j}
\frac{\partial_{x_j} S_i}{\partial_{x_i} S_i}
=
\frac{f}{\rho_{\nu^{\bot}_{x,i}}} .
\end{equation}
Consider the following formal computations
\begin{align*}
\int \bigl( \partial_{x_j} S_i \bigr)^2 \ d\nu
& = 
\int \bigl( \partial_{x_i} S_i \bigr)^2 \frac{f^2}{\rho^2_{\nu^{\bot}_{x,i}}} \ d\nu
=
\int  \frac{\partial_{x_i} S_i}{\rho_{\gamma}(S_i)}  \frac{f^2}{\rho_{\nu^{\bot}_{x,i}}} \ d\nu 
\\& =
\int \int \frac{\partial_{x_i} S_i}{\rho_{\gamma}(S_i)}  f^2 \ d\nu_{i-1} \ dx_i
=
-
\int  \int \frac{\partial_{x_i} S_i}{\rho_{\gamma}(S_i)} S^2_i f^2 \ d\nu_{i-1} \ dx_i \\&
-
2\int  \int \frac{ S_i}{\rho_{\gamma}(S_i)}  f f_{x_i} \ d\nu_{i-1} \ dx_i
\le \int \int \frac{f^2_{x_i}}{\rho_{\gamma}(S_i)\partial_{x_i} S_i}   \ d\nu_{i-1} \ dx_i \\&
=
\int \Bigl( \frac{f_{x_i}}{\rho_{\nu^{\bot}_{x,i}}} \Bigr)^2  \ d\nu_{i}
=
\int \Bigl( \frac{ \partial_{x_j} \rho_{\nu_i}}{ \rho_{\nu_i}} -
\frac{
\int_{-\infty}^{\infty} \partial_{x_j} \rho_{\nu_i}(x,t) \ dt }{\int_{-\infty}^{\infty} \rho_{\nu_i}(x,t) \ dt}
\Bigr)^2  \ d\nu 
\\& =
\int \Bigl( \Exp^{\mathcal F_i}_{\nu} \beta_j -
\Exp^{\mathcal F_{i-1}}_{\nu} \beta_j \Bigr)^2  \ d\nu
=
\int \Bigl( \Exp^{\mathcal F_i}_{\nu} \beta_j -
\Exp^{\mathcal F_{i-1}}_{\nu} \beta_j \Bigr)^2  \ d\nu
\\& =
\int \Bigl( \Exp^{\mathcal F_i}_{\nu} \beta_j  \Bigr)^2  \ d\nu
-
\int \Bigl( \Exp^{\mathcal F_{i-1}}_{\nu} \beta_j  \Bigr)^2  \ d\nu.
\end{align*}
To justify the global integration above we integrate again with respect to $\xi \cdot \nu$, where $\xi$ is a compactly supported smooth positive  function on $\mathbb{R}^d$.  
Repeating the above arguments one gets
\begin{align*}
\int \bigl( \partial_{x_j} S_i \bigr)^2 \xi \ d\nu
& =
- \int S_i \cdot \partial_{x_j} S_i \cdot \partial_{x_i} \xi \ d \nu
- \int  \int \frac{\partial_{x_i} S_i}{\rho_{\gamma}(S_i)} S^2_i f^2  \xi \ d\nu_{i-1} \ dx_i 
\\&
-
2\int  \int \frac{ S_i}{\rho_{\gamma}(S_i)}  f f_{x_i}  \xi \ d\nu_{i-1} \ dx_i
\end{align*}
The term $-\int S_i \cdot \partial_{x_j} S_i \cdot \partial_{x_i} \xi \ d \nu$
can be estimated by 
$$
\varepsilon \int (\partial_{x_j} S_i)^2 \xi \ d \nu + \frac{4}{\varepsilon} \int S^2_i \bigl( \frac{\partial_{x_i} \xi}{\xi}\bigr)^2 \xi \ d \nu.
$$
Finally
$$
(1-\varepsilon) 
\int \bigl( \partial_{x_j} S_i \bigr)^2 \xi \ d\nu
\le 
\int \Bigl( \Exp^{\mathcal F_i}_{\nu} \beta_j -
\Exp^{\mathcal F_{i-1}}_{\nu} \beta_j \Bigr)^2 \xi \ d\nu
+ \frac{4}{\varepsilon} \int S^2_i \bigl( \frac{\partial_{x_i} \xi}{\xi}\bigr)^2 \xi \ d \nu.
$$
 Estimating the term $\int S^2_i \bigl( \frac{\partial_{x_i} \xi}{\xi}\bigr)^2 \xi \ d \nu$ by the H{\"o}lder inequality and choosing an appropriate sequence $\xi_n \to 1$ we complete
the justification of the above formal computation.

It remains to approximate an arbitrary density $\rho$ on $\mathbb{R}^d$ with $\int \beta^2_i \rho \ dx < \infty$
by smooth densities and prove that the desired a-priori estimate is preserved under taking the limit. Indeed, let  $\rho_{(k)} = e^{-V_k}$ be approximating densities constructed in Lemma
\ref{approx}. Let $S^{(k)}$ be the triangular mappings pushing forward $\rho_{(k)} \ dx$
onto $\gamma$.
Note that the functions $S^{(k)}_i \cdot \rho_{(k)}$ are in $W^{1,1}(\mathbb{R}^n)$. Indeed, 
$$
\int \|D  S^{(k)}\|_{HS} \rho_{(k)} \ dx \le \Bigl[  \int \|D S^{(k)}\|^2_{HS} \rho_{(k)} \ dx  \Bigr]^{1/2} \le \sum_{i} \int (\beta^{(k)}_i)^2 \rho_{(k)}  
 \ dx \le \| \sqrt{\rho_{(k)}}\|^2_{W^{1,2}(\mathbb{R}^n)}
$$
and
$$
\int | S^{(k)}_i| \| \nabla \rho_{(k)}\| \ dx \le \Bigl[ \int |S^{(k)}_i|^2 \rho_{(k)} \ d x \Bigr]^{1/2} \| \sqrt{\rho_{(k)}}\|_{W^{1,2}(\mathbb{R}^n)} =  \Bigl[ \int x^2_i \ d \gamma \Bigr]^{1/2} \| \sqrt{\rho_{(k)}}\|_{W^{1,2}(\mathbb{R}^n)}.
$$
Using Sobolev embeddings and extracting an almost everywhere convergent subsequence one can assume from the very beginning that $S^{(k)}\cdot \rho_{(k)}$ converges almost everywhere. Using that $\rho_{(k)}$ converges
almost everywhere to $\rho$, one can easily see that $S^{(k)}$ converges to a triangular mapping $S$ 
at $\nu$-almost all points. Using almost everywhere convergence it is easy to check that $S$ pushes forward $\rho$ onto $\gamma$. From the  almost everywhere convergence and the following change of variables formula 
$$
\int \|S^{(k)}\|^2 \rho_{(k)} \ d x = \int \|x\|^2 \ d \gamma  = \int \|S\|^2 \rho \ d x, 
$$
one gets that $S^{(k)}\cdot \sqrt{\rho_{(k)}}$ converges to  $S\cdot \sqrt{\rho}$ in $L^2(\mathbb{R}^d)$. Applying the estimates above  one   proves that  $\nabla S^{(k)}_i\cdot \sqrt{\rho_{(k)} }$
converges (up to a subsequence) weakly in $L^2(\mathbb{R}^n)$ to a vector field $v$ . Standard integration by parts arguments show that $v$ can be identified with $\nabla S_i\cdot \sqrt{\rho}$. Indeed,
$$
\int \varphi \cdot \partial_{x_i} S ^{(k)} {\rho^{(k)}} \ dx = - \int \partial_{x_i} \varphi \cdot   S^{(k)} \rho^{(k)} \ dx  - \int  \varphi  \ S^{(k)} \cdot  {\partial_{x_i} \rho^{(k)}} \ dx.
$$
The left-hand side converges to $\int \varphi \cdot v \sqrt{\rho^{}} \ dx $ and the right-hand side to $$ - \int \partial_{x_i} \varphi \cdot   S^{} \rho^{} \ dx  - \int  \varphi  \ S^{} \cdot  {\partial_{x_i} \rho^{}} \ dx$$
(this follows from the strong convergence of $S^{(k)}\cdot \sqrt{\rho_{(k)}}$  and $\nabla{\rho_{(k)}}/\sqrt{\rho_{(k)}}$).
Hence
$$
\int |\partial_{x_i} S|^2 \rho \ dx \le \underline{\lim}_k  \int |\partial_{x_i} S^{(k)}|^2 \rho_{(k)} \ dx \le  \lim_k \sum_{i} \int (\beta^{(k)}_i)^2 \rho_{(k)} \ dx= \sum_{i} \int (\beta_i)^2 \rho \ dx. 
$$
The other estimates can be justified in the same way. Hence the proof is complete.
\end{proof}

\begin{remark}
It is clear, that formula  (\ref{onedim-cv}) remains true in the non-smooth setting, for instance under the assumptions of Proposition \ref{l2triest}. We understand $\partial_{x_i} S_i$ as the Sobolev derivative or
just as the classical  derivative of the one-dimensional increasing mapping $x_i \to S_i$. Taking product from $i=1$ to $d$ in (\ref{onedim-cv})  we obtain the change of variables formula
\begin{equation}
\label{ch-var}
\rho_{\nu} =  (2 \pi)^{-d/2} e^{-\frac{1}{2} |S|^2}\det DS =  \prod_{i=1}^{d} \rho_{\gamma}(S_i) \cdot  \partial_{x_i} S_i.
\end{equation}
\end{remark}

\begin{remark}
In what follows we will give a proof for a-priori estimates only in the case of { smooth and positive densities}. The complete justification for Sobolev densities
can be spelt out as in the proof of Proposition \ref{l2triest}. 

In particular, note that since all the densities are positive and smooth, all the expressions in the intermediate computations  are well-defined.

We also note that in the general (i.e. Sobolev) case $\partial_{x_i} S_i$ remains positive $\nu$-{almost everywhere}, because $\partial_{x_i} S_i=0$ implies that the corresponding conditional density of $\nu$ vanishes, which can happen only on a set  
of $\nu$-measure zero.
\end{remark}

\begin{remark}
\label{int-ident}
Another estimate of this type has been mentioned (without rigorous proof)  in  \cite{Kol2010}
$$  
\int \beta_i^2 \  d \nu = \int \|\partial_{x_i} S\|^2 \ d \nu +
\sum_{k} \int \Bigl( \frac{\partial_{x_i x_k} S_k}{\partial_{x_k} S_k}
 \Bigr)^2 \ d \nu.
$$
Moreover, if the image measure $\mu$ is not Gaussian, but uniformly log-concave, i.e. has the form
$\mu = e^{-W} \ dx$ with $D^2 W \ge K \cdot \mbox{\rm{Id}}$, $K>0$, then
$$
\int \beta_i^2 \  d \nu \ge K \int \|\partial_{x_i} S\|^2 \ d \nu +
\sum_{k} \int \Bigl( \frac{\partial_{x_i x_k} S_k}{\partial_{x_k} S_k}
 \Bigr)^2 \ d \nu.
$$
\end{remark}

\begin{remark}
\label{lptriest}
One can easily generalize Proposition \ref{l2triest} to the $L^p$-case. Under the same assumptions for every
$p >1$ there exists $C=C(p)$ such that
$$
\int \bigl( \partial_{x_i} S_i \bigr)^p  d\nu
\le
  C(p)  \int \bigl|\Exp^{\mathcal{F}_i}_{\nu} \beta_i\bigr|^p   d\nu
$$
and
$$
\int \bigl| \partial_{x_j} S_i \bigr|^p  d\nu
\le C(p)
\int \Bigl| \Exp^{\mathcal F_{i}}_{\nu} \beta_j   - \Exp^{\mathcal F_{i-1}}_{\nu} \beta_j  \Bigr|^p   d\nu.
$$
The proof follows along the same line of arguments as above.
\end{remark}

We prove  some $L^p$-estimates for higher order derivatives.
Taking logarithm of both sides of the identity
$
\frac{\rho_{\nu_{x,i}^{\bot}}}{\rho_{\gamma}(S_i)}
=
\partial_{x_i} S_i
$
and differentiating the result  along $x_j$ one gets
$$
\frac{\partial_{x_j} \rho_{\nu_{x,i}^{\bot}} }{\rho_{\nu_{x,i}^{\bot}}}
+ S_i \cdot \partial_{x_j} S_i = \frac{\partial_{x_i x_j} S_i}{\partial_{x_i} S_i}.
$$
Hence
$$
\partial_{x_i x_j} S_i
=
\partial_{x_i} S_i \frac{\partial_{x_j} \rho_{\nu_{x,i}^{\bot}} }{\rho_{\nu_{x,i}^{\bot}}}
+ S_i \cdot \partial_{x_j} S_i \cdot \partial_{x_i} S_i.
$$
Then applying the standard H{\"o}lder and Jensen  inequalities and using that $S_i \in L^{N}(\eta)$ for
every $N >0$, we get trivially the following bound.

\begin{proposition}
\label{Siij}
For every $p > 1$ and $\varepsilon>0$  there exists $C(p,\varepsilon)$ such that under assumptions that $\beta_k \in L^p(\nu)$ for all $k$,
one has 
$$
\Bigl\|\frac{\partial_{x_i x_j} S_i}{\partial_{x_i} S_i}\Bigr\|_{L^p(\nu)} 
\le C(p, \varepsilon)  \|\beta_j\|_{L^{p+\varepsilon}(\nu)}
$$
and
$$
\| \partial_{x_i x_j} S_i \|_{L^p(\nu)} \le 
C(p,\varepsilon) \ \| |\beta_i|^p \|_{L^{2+\varepsilon}(\nu)}  \cdot \| |\beta_j|^p \|_{L^{2+\varepsilon}(\nu)} .
$$
See also Remark \ref{int-ident}.
\end{proposition}

It remains to estimate $\| \partial_{x_j x_m} S_i \|_{L^p(\nu)}$ for $j \ne i, m \ne i$.

\begin{proposition}
\label{Sjmi}
Let  $j <  m < i$  and $p > 1$. Assume that $\beta_j, \beta_m \in L^{2p}(\nu)$ and $\beta_j$
admits partial Sobolev derivative $\partial_{x_m}\beta_j \in L^{p}(\nu)$.

Then  there exists  $C(p)$ such that
$$
 \int | \partial_{x_j x_m} S_i |^p \ d \nu  \le C(p) \int \Bigl(  \beta^{2p}_j + \beta^{2p}_m +  |\partial_{x_m}\beta_j|^p \Bigr) \ d \nu.
$$
\end{proposition}
\begin{proof} 
In the same way as above the proof is reduced to the case where the densities are smooth and positive and admits integrable derivatives (see Proposition \ref{l2triest}). 

For simplicity let us consider only the case $p=2$.
We  use relation (\ref{i-j}):
$
{\partial_{x_j} S_i}
= {\partial_{x_i} S_i}
\frac{f}{\rho_{\nu^{\bot}_{x,i}}} 
$ with 
\begin{equation}
\label{x-ij}
f=
\frac{\int_{-\infty}^{x_i} \partial_{x_j} \rho_{\nu_i}(x,t) \ dt}{\int_{-\infty}^{\infty} \rho_{\nu_i}(x,t) \ dt}
-
\frac{\int_{-\infty}^{x_i}  \rho_{\nu_i}(x,t) dt \
\int_{-\infty}^{\infty} \partial_{x_j} \rho_{\nu_i}(x,t) \ dt \cdot }{\Bigl(\int_{-\infty}^{\infty} \rho_{\nu_i}(x,t) \ dt
\Bigr)^2}.
\end{equation}
Differentiating  $
{\partial_{x_j} S_i}
= {\partial_{x_i} S_i}
\frac{f}{\rho_{\nu^{\bot}_{x,i}}} 
$  along $x_m$ we get
\begin{align*}
{\partial_{x_j x_m} S_i} 
= &
 {\partial_{x_i x_m} S_i}
\frac{f}{\rho_{\nu^{\bot}_{x,i}}} + 
{\partial_{x_i} S_i}
\cdot \frac{\partial}{\partial x_m} \Bigl[ \frac{f}{\rho_{\nu^{\bot}_{x,i}}} \Bigr] 
= 
\partial_{x_j} S_i  \frac{ {\partial_{x_i x_m} S_i}  }{\partial_{x_i} S_i} + 
{\partial_{x_i} S_i}
\cdot \frac{\partial}{\partial x_m} \Bigl[ \frac{f}{\rho_{\nu^{\bot}_{x,i}}} \Bigr]
\\&
=
\partial_{x_j} S_i 
\frac{ {\partial_{x_i x_m} S_i}}{\partial_{x_i} S_i} - 
{\partial_{x_j} S_i}
 \frac{ \partial_{x_m}  \rho_{\nu^{\bot}_{x,i}} }{\rho_{\nu^{\bot}_{x,i}}}
+  {\partial_{x_i} S_i}
 \frac{ {\partial_{x_m}} f}{\rho_{\nu^{\bot}_{x,i}}}.
\end{align*}
The bounds for the first two terms follow immediately from  the previous estimates.
Let us estimate  
$ 
\int \Bigl( {\partial_{x_i} S_i}
 \frac{ {\partial_{x_m}} f}{\rho_{\nu^{\bot}_{x,i}}} \Bigr)^2 
$. One has
\begin{align*}
&
\int \bigl( \partial_{x_i} S_i \bigr)^2 \frac{f^2_{x_m}}{\rho^2_{\nu^{\bot}_{x,i}}} \ d\nu
=
\int  \frac{\partial_{x_i} S_i}{\rho_{\gamma}(S_i)}  \frac{f^2_{x_m}}{\rho_{\nu^{\bot}_{x,i}}} \ d\nu 
 =
\int \int \frac{\partial_{x_i} S_i}{\rho_{\gamma}(S_i)}  f^2_{x_m} \ d\nu_{i-1} \ dx_i
\\& =
-
\int  \int \frac{\partial_{x_i} S_i}{\rho_{\gamma}(S_i)} S^2_i f^2_{x_m} \ d\nu_{i-1} \ dx_i
-
2\int  \int \frac{ S_i}{\rho_{\gamma}(S_i)}  f_{x_m} f_{x_i x_m} \ d\nu_{i-1} \ dx_i \\&
\le \int \int \frac{f^2_{x_i x_m}}{\rho_{\gamma}(S_i)\partial_{x_i} S_i}   \ d\nu_{i-1} \ dx_i 
=
\int \Bigl( \frac{f_{x_i x_m}}{\rho_{\nu^{\bot}_{x,i}}} \Bigr)^2  \ d\nu_{i}.
\end{align*}
Differentiating (\ref{x-ij}) one gets
\begin{align*}
\frac{f_{x_i x_m}}{\rho_{\nu^{\bot}_{x,i}}}  & = \frac{\partial_{x_m x_j} \rho_{\nu_i}}{\rho_{\nu_i}}
-
\frac{\partial_{x_j} \rho_{\nu_i}}{\rho_{\nu_i}} \frac{\int_{-\infty}^{\infty} \partial_{x_m}  \rho_{\nu_i}(x,t) dt \
}{\int_{-\infty}^{\infty} \rho_{\nu_i}(x,t) \ dt }
-
\frac{\partial_{x_m} \rho_{\nu_i}}{\rho_{\nu_i}} \frac{\int_{-\infty}^{\infty} \partial_{x_j}  \rho_{\nu_i}(x,t) dt \
}{\int_{-\infty}^{\infty} \rho_{\nu_i}(x,t) \ dt }
\\&  -
\frac{\int_{-\infty}^{\infty} \partial_{x_j x_m}  \rho_{\nu_i}(x,t) dt \
}{\int_{-\infty}^{\infty} \rho_{\nu_i}(x,t) \ dt }
+
2 \frac{\int_{-\infty}^{\infty} \partial_{x_m}  \rho_{\nu_i}(x,t) dt \
\int_{-\infty}^{\infty} \partial_{x_j} \rho_{\nu_i}(x,t) \ dt  }{\Bigl(\int_{-\infty}^{\infty} \rho_{\nu_i}(x,t) \ dt \Bigr)^2}.
\end{align*}
Arguing as above, we easily get
$$
\Exp \Bigl(  \frac{f_{x_i x_m}}{\rho_{\nu^{\bot}_{x,i}}} \Bigr)^2
\le C  \int \Bigl(  \beta^4_j + \beta^4_m +  (\partial_{x_m}\beta_j)^2 \Bigr) \ d \nu.
$$
Hence the  proof is complete.
\end{proof}

\section{Transfer of solutions}

We consider in this section a probability measure 
$\nu$ on the space $X$, where  $X=\mathbb{R}^d$ or $X=\mathbb{R}^{\infty}$.
We denote by $\gamma$ the standard Gaussian measure if $X = \mathbb{R}^d$
and the 
product of the standard Gaussian measures on $\mathbb{R}^1$
$$
\gamma = \prod_{i=1}^{\infty} \gamma_i(dx_i),
$$
if $X=\mathbb{R}^{\infty}$.

Everywhere in this section $S$ is the  triangular mapping pushing forward  $\nu$ onto $\gamma$. 
 As usual, we set: $T = S^{-1}$ and  $c = DS(T) \cdot b(T)$.  

It will be assumed throughout that $\nu$ admits logarithmic derivatives $\beta_i \in L^p(\nu)$, $i \in \mathbb{N}$, {at least} for some $p>1$ (independent on $i$).
Thus by Remark \ref{lptriest} the functions $S_i$ are all Sobolev, more precisely $S_i \in L^{n}(\nu)$ for all $n$ and $|\nabla S_i| \in L^{p}(\nu)$.

We  also apply systematically the 

{\bf Chain rule:} for every $f \in W^{1,p}(\nu)$ and every smooth compactly supported function $\varphi$ on $\mathbb{R}$ one has $\varphi(f) \in W^{1,p}(\nu)$ (see Lemma 2.6.9 \cite{B2008}).

We also need the following important fact (see Theorem  2.6.11 \cite{B2008}).
\begin{theorem}
\label{RoZh}
Assume that $d<\infty$, $p \ge 1$. The set of smooth compactly supported functions is dense in the weighted Sobolev space $W^{1,p}(\nu)$, $\nu = \rho \ dx$ provided
$\log {\rho} \in W^{1,p}(\nu)$.
\end{theorem}

Everywhere below $\beta_i$ is the logarithmic derivative of $\nu$ along $e_i$.

{  The section is divided in two subsections: finite-dimensional and infinite-dimensional statements. We keep the same notations $\nu, \gamma,\beta_i,  T, S$ etc. for both situations.}

\subsection{Finite-dimensional estimates}

\begin{lemma}
\label{core-trans}
Assume that $X = \mathbb{R}^d$ and $\beta_i \in L^p(\nu)$,  for any $p \ge 1$ and all $i \in \{1,\cdots,d\}$.
Assume, in addition, { that $\frac{1}{\rho_{\nu}} \in \cap_{p \ge 1} L^p_{loc}(\mathbb{R}^d)$}. Then for every $\varphi \in C^{\infty}_0(\mathbb{R}^d)$ 
the function $\varphi(T)$ belongs to $W^{1,n}(\gamma)$ for any $n \in \mathbb{N}$.
\end{lemma}
\begin{proof}
{ 
To show   that $\varphi(T)$ belongs to $W^{1,n}(\gamma)$ we apply the chain rule and  Remark \ref{lptriest}. According
to this Remark  $\langle DS^* e_i,e_j \rangle \in \cap_{p \ge 1} L^p(\nu)$ for every $1 \le i,j \le d$. One has
$$
\int \| \nabla \bigl[\varphi(T)\bigr] \bigr] \|^n \ d \gamma \le  \int \| DT^* \cdot \nabla \varphi(T)\|^n\ d \gamma
=  \int \| (DS^* )^{-1}\cdot \nabla \varphi\|^n \ d \nu.
$$
It remains to show that all the functions $ \langle (DS^* )^{-1}e_i,e_j \rangle$  belong to $\cap_{p \ge 1} L^p_{loc}(\nu) $. Since $\langle DS^* e_i,e_j \rangle$ admits the same property, we only need to show the local integrability of
$\frac{1}{(\det S)^n}$ for any $n \in \mathbb{N}$. Taking into account that $\frac{1}{\det S} = \frac{\rho_{\gamma}(S)}{\rho_{\nu}}$ (see formula (\ref{ch-var}) for the exact expression), the boundedness of $\rho_{\gamma}$ , and the assumptions of this lemma we immediately get the claim.
}
\end{proof}

\begin{lemma}
\label{btoc0}
Assume that  $X = \mathbb{R}^{d}$  and $\beta_i \in L^{p}(\nu)$  for some $p >1$ and all $i$. 
Then $\varphi(S) \in W^{1,p}(\nu)$ for every $\varphi \in C^{\infty}_0(\mathbb{R}^d)$.
\end{lemma}
\begin{proof}
Apply the chain rule  and Remark \ref{lptriest}. 
\end{proof}

\subsection{Infinite-dimensional estimates}

In this subsection we apply Lemma \ref{core-trans}, Lemma \ref{btoc0} in the infinite-dimensional case to functions of the type 
$\varphi(T), \varphi(S)$, where $\varphi$ depends on finite number of coordinates.

\begin{lemma}
\label{btoc}
Assume that  $X = \mathbb{R}^{\infty}$  and $\beta_i \in L^{p}(\nu)$,  $b_i \in  L^{p^*}(\nu) $ for some $p >1$ and all $i$. 
Assume that $\mbox{\rm{div}}_{\nu} b  \in L^1(\nu)$. Then   every $c_i \in L^1(\gamma)$ and, in addition, $c$  admits a divergence and the following relation holds:
$\mbox{\rm{div}}_{\gamma} c \circ S = \mbox{\rm{div}}_{\nu} b $.
\end{lemma}
\begin{proof}
{ Let us fix $n \in \mathbb{N}$ and  take a smooth cylindrical function $ \varphi = \varphi(x_1, \cdots, x_n)$
 (with compact support, if we  consider $\varphi$  as a function on $\mathbb{R}^{n}$).
It is easy to see that 
the function $\varphi(S)$ is  cylindrical.  Thus one can consider $\varphi(S)$ as a finite-dimensional function with the reference measure $\nu_n$ on $\mathbb{R}^n$.
Since the logarithmic derivatives of $\nu_n$ are obtained as conditional expectations of functions $\beta_i$, they all  belong to $L^{p}(\nu_n)$ by the Jensen's inequality.
 Thus it follows from Lemma \ref{btoc0} that $\varphi(S) \in W^{1,p}(\nu_n)$, Considering $\varphi(S)$ as a  cylindrical function on $\mathbb{R}^{\infty}$ we immediately get  $\varphi(S) \in W^{1,p}(\nu)$.
Clearly, the finite-dimensional change rule implies} 
$
 \nabla \varphi(S) = (DS)^* \nabla \varphi(S).$
By Theorem \ref{RoZh} there exists a sequence of $C^{\infty}_0(\mathbb{R}^n)$-functions $\{\psi_k\}$
such that $\psi_k \to \varphi(S)$ and $\nabla \psi_k \to \nabla \bigl[ \varphi(S)\bigr]$  in $L^p(\nu)$.
This, in particular, implies that the relation
$- \int f \cdot \mbox{\rm{div}}_{\nu} b \ d \nu = \int \langle \nabla f, b \rangle \ d \nu $ holds for  $f = \varphi(S)$. The assumptions of this lemma now imply that $c_i \in L^1(\gamma)$ for every $i$. 
Hence
\begin{align*}
\int \langle \nabla \varphi, c \rangle \ d \gamma   & =\int \langle \nabla \varphi, c \rangle \circ S \ d \nu  = \int \langle (DS^*)^{-1} \nabla (\varphi(S)), c(S) \rangle \ d \nu
\\& =
\int \langle \nabla (\varphi(S)), b \rangle \ d \nu
= - \int \varphi(S) \mbox{\rm{div}}_{\nu} b \ d \nu
= - \int \varphi \ \mbox{\rm{div}}_{\nu}  b(T) \ d \gamma.
\end{align*}
{ By linearity this identity can be extended to any $\varphi \in \mathcal{C}$. The latter means that  $\mbox{\rm{div}}_{\gamma} c  = \mbox{\rm{div}}_{\nu} b \circ T$.}
\end{proof}

\begin{proposition}
\label{gauss-ex}
Assume that $X =\mathbb{R}^{\infty}$.
Let  $\rho(t,x)$ be a solution to the equation
$\dot{\rho} + \mbox{\rm{div}}_{\nu} (\rho \cdot b)=0$. 
  Assume that there exists $p > 1$ such that 
$\sup_{t \in [0,T]} \| \rho(t, \cdot) \cdot b_i \|_{L^{p^*}(\nu)} < \infty$ and, in addition,
$\beta_i \in L^{p}(\nu)$ for all $i$.

Then the function $g(t,x)$ defined by the
relation ${g} \cdot \gamma = (\rho \cdot \nu) \circ S^{-1}$ is the solution to the equation
$$
\dot{g} + \mbox{\rm{div}}_{\gamma} (g \cdot c)=0
$$
\end{proposition}
\begin{proof}
We know that
$$
\int \varphi \rho(t,x)\ d \nu
=
\int \varphi \rho(0,x) \ d \nu
+
\int_{0}^{t} \int \langle {b}, \nabla \varphi \rangle  \rho(s,x) \ d\nu \ ds, \ \ \ \mbox{for all} \ t \in [0,T].
$$
{Take a smooth cylindrical function $ \psi = \psi(x_1, \cdots, x_n)$
 (with compact support, if we  consider $\psi$  as a function on $\mathbb{R}^{n}$).
 Let us apply the above identity to the function $\varphi = \psi(S)$. This is possible, because 
$\varphi \in W^{1,p}(\nu)$  (see the proof of  Lemma \ref{btoc}) and $\sup_{t \in [0,T]} \| \rho(t, \cdot) \cdot b_i \|_{L^{p^*}(\nu)} < \infty$.} 

By the change of variables formula
$$
\int \varphi \rho(t,x)\ d \nu
= \int \psi  g(t,x) \ d \gamma, \ \  
\int \varphi \rho(0,x) \ d \nu = \int \psi  g(0,x) \ d \gamma.
$$
Taking into account the chain rule $\nabla \varphi= (DS)^* \nabla \psi(S)$ we immediately get for all $t \in [0,T]$
$$
\int_{0}^{t} \int \langle {b}, \nabla \varphi \rangle  \rho(s,x) \ d\nu \ ds = \int_{0}^{t} \int \langle DS \cdot {b},  \nabla \psi(S) \rangle  \rho(s,x) \ d\nu \ ds = \int_{0}^{t} \int \langle c ,  \nabla \psi \rangle  \ g(s,x) \ d\gamma \ ds.
$$
Hence $g$ satisfies the desired integral relation and
the proof is complete.
\end{proof}

\begin{proposition}
\label{nu-ex}
Assume that $X =\mathbb{R}^{\infty}$ and the following assumptions hold
\begin{itemize}
\item[1)]
 $\beta_i \in L^m(\nu)$,  for  all $m, i \in \Nat$
\item[2)]
$\frac{1}{\rho_{\nu_n}}$ is locally integrable in any power for every $n \in \Nat$, where
$\rho_{\nu_n}$ is 
the Lebesgue density  of the projection $\nu \circ P^{-1}_n = \rho_{\nu_n} \ dx$
\end{itemize}
Assume, in addition, that that $g $ solves the equation
$\dot{g} + \mbox{\rm{div}}_{\gamma} (g \cdot c)=0$ for some $c$ satisfying 
$\sup_{t \in [0,T]} \| g(t, \cdot) \cdot c_i \|_{L^{1+\varepsilon}(\gamma)} < \infty$ for some $\varepsilon>0$ and  all $i$.

 Then the function $\rho$ defined by the
relation $\rho \cdot \nu = (g \cdot \gamma) \circ T^{-1}$ is the solution to the equation
$$
\dot{\rho} + \mbox{\rm{div}}_{\nu} (\rho \cdot b)=0.
$$
\end{proposition}
\begin{proof}
We apply the same arguments as in the proof of the previous proposition. 
We note that $\int_{0}^{t} \int \langle {b}, \nabla \varphi \rangle  \rho(s,x) \ d\nu \ ds$
is  well-defined for every $\varphi \in \mathcal{C}$, $t \in [0,T]$,  because  $ b_i = \langle (DS)^{-1} c(S), e_i \rangle$, any function $\langle (DS)^{-1} e_i, e_j \rangle$ (depending on a  finite number of variables)
is locally integrable in any power by the previous proposition and $$\sup_{0 \le s \le T} \|c_i(S) \rho(s,\cdot)\|_{ L^{1+\varepsilon}(\nu)} < \infty$$ by the change of variables formula. 

The relation 
$$
\int_{0}^{t} \int \langle c,  \nabla \psi \rangle  g(s,x) \ d \gamma \ ds   = \int_{0}^{t} \int \langle {b}, \nabla \varphi \rangle  \rho(s,x) \ d\nu \ ds, \ t \in [0,T], 
$$
with $\psi=\varphi(T)$, $\varphi \in \mathcal{C}$,
 can be easily justified with the help of  Lemma \ref{core-trans}. 
{Indeed, since $\psi$
depends on a finite number of variables,  Lemma \ref{core-trans} implies that $\psi$
belongs to $W^{1,n}(\gamma)$ for any $n \in \mathbb{N}$ (see also the proof of  Lemma \ref{btoc}).
Hence $\psi$ can be approximated in the corresponding Sobolev norm by functions from $\mathcal{C}$ and the desired relation is justified.}
\end{proof}

\section{Existence}

In this section we prove the existence result by transferring a solution in the  Gaussian case 
(whose existence was established in 
\cite{AF}) with the help of a triangular mapping. 

\begin{theorem}
\label{exist-th}
Assume that $\nu$ is a probability measure on $\mathbb{R}^{\infty}$  such that:
\begin{itemize}
\item[1)]
$$
\beta_i \in L^m(\nu) \ \ \ \mbox{for all} \ m,i \in \Nat;
$$
\item[2)]
there exists $p>1$ such that
$$
b_i \in L^{p}(\mu) \ \ \ \mbox{for all} \ i \in \Nat;
$$
\item[3)] there exists $\varepsilon>0$ such that
$$
\exp\bigl(\varepsilon(\mbox{\rm div}_{\nu} b)_{-} \bigr)
\in
L^1(\nu);
$$
\item[4)]
$\frac{1}{\rho_{\nu_n}}$ is locally integrable in any power for every $n \in \Nat$, where
$\rho_{\nu_n}$ is 
the Lebesgue density  of the projection $\nu \circ P^{-1}_n = \rho_{\nu_n} \ dx$.
\end{itemize}
Then for every $\rho_0 \in L^{q'}(\nu)$, and $\tilde{q}$ with $q' > \tilde{q} > p^*$ there exists $t_0 >0$ depending on the above parameters such that the equation
$$
\dot{\rho} + \mbox{\rm div}_{\nu} \bigl( b \cdot \rho \bigr) =0
$$
has a solution  on $[0,t_0]$ satisfying $\rho|_{t=0}=\rho_0$
and
\begin{equation}
\label{U}
\sup_{t \in [0,t_0]} \| \rho(t,\cdot)\|_{L^{\tilde{q}}} < \infty .
\end{equation}
\end{theorem}

\begin{remark}
One can easily see that assumptions 1), 4) together with Sobolev embedding  imply that $1/\rho_{\nu_n}$ is H{\"o}lder continuous. This may be sometimes restrictive for applications.
We stress that we need  1) and 4) mainly for a-priori estimates on $DT$ (see Lemma \ref{core-trans}). There are some possibilities to weaken these assumption. 
Some (weaker) sufficient conditions for  $T$  to be locally Sobolev  one can find in \cite{ZhOv}.  
This result is applicable if one has high integrability of $\rho_0$ and $b_i$.
Some bounds on  $DT$  are available under the assumption that $\nu$ is log-concave. They work even better if instead of triangular mapping one applies optimal transportation.
See Theorem \ref{logconcave} and Example \ref{fin-dim-ex} below.
\end{remark}

\begin{proof}
Consider the triangular mapping $T$ sending $\gamma$ to  $\nu$.
Let us show that $c=DS(T) \cdot b(T)$ satisfy all the assumptions
of  Lemma \ref{gauss-eq}.
One has
$$
c_i = \sum_{j=1}^{i} \frac{\partial S_i}{\partial x_j} (T) b_j(T).
$$
It follows immediately from the assumptions of this theorem and Remark \ref{lptriest} that $c_i \in L^{p'}(\nu)$
for every $i$ and $p'<p$.

By Lemma \ref{btoc}  $\mbox{div}_{\gamma} c \circ S = \mbox{div}_{\nu} b$. Consequently, the assumption 2) of Lemma
\ref{gauss-eq} is satisfied. Hence, there exists a solution to the equation
$\dot{g} + \mbox{div}_{\gamma}(g \cdot c) =0$ with $g(0,x) = \rho(0, T(x))$.
Proposition \ref{nu-ex} now implies that $\rho=g(S)$ is the desired solution.

Property (\ref{U}) is a slight extension of the corresponding statement of Lemma \ref{gauss-eq} and can be easily checked.
Hence the proof is complete.
\end{proof}

\begin{example}
\label{Gibbs}
Let us give an example of a probability measure on $\mathbb{R}^{\infty}$ with integrable logarithmic derivatives
which is typical for applications and satisfies the assumption of the above theorem.
We consider  Gibbs  measures  on
a lattice $\R^{\Z^d}$,
 which can be formally written in  the following way
\begin{equation}
\label{Gibbs-potential}
\nu_* = '' \exp\Bigl( -\sum_{k \in \Z^d} V_k(x_k)   -\sum_{k, j \in \Z^d} W_{k,j}(x_k,x_j)  \Bigr) \ dx'',
\end{equation}
where $''dx''$ denotes infinite-dimensional Lebesgue measure on $\mathbb{R}^{\mathbb{Z}^d}$
(which does not exist).
The following existence result has been established in \cite{AKRT}.
Assume that there exist a number $N \ge 2$ and a symmetric matrix $J=\{J_{k,j}\}_{k,j \in \Z^{\R^d}}$ such that
$$
W_{k,j}(x_k,x_j) = W_{j,k}(x_j,x_k)
$$
$$
| W_{k,j}(x_k,x_j) | \le J_{k,j} (1+ |x_k| + |x_j|)^N
$$
$$
| \partial_{x_k}  W_{k,j}(x_k,x_j) | \le J_{k,j} (1+ |x_k| + |x_j|)^{N-1}
$$
$$
|V_k(x_k)| \le C(1+|x_k|)^L, \ \ 
|\partial_{x_k} V_k(x_k)| \le C(1+|x_k|)^{L-1}
$$
$$
\partial_{x_k} V_k(x_k) \cdot x_k \ge A|x_k|^{N + \sigma} -B
$$
for some $A, B,C, \sigma > 0, L \ge1$.

The matrix $J$ is also assumed to be { fast} decreasing  (see \cite{AKRT} for details), in particular the finite range case 
 $J_{k,j}=0$ if $|j-k|>N_0$ for some $N_0$ is included.
Then there exists a  probability  ("Gibbs") measure $\nu$ on $\R^{\Z^d}$  with exponentially integrable logarithmic derivatives
$$
\beta_k = \partial_{x_k} V_k(x_k) + \sum_{j \in \Z^d}  \partial_{x_k} ( W_{k,j} + W_{j,k})á \ \ k \in \mathcal{Z}^d. 
$$
It was shown in \cite{AKRT} that such $\nu$ is a rigorous definition of the measure $\nu_*$ in (\ref{Gibbs-potential})
via the Dobrushin-Lanford-Ruelle equations. 
See  \cite{AKRT2} for uniqueness results.
\end{example}

\section{Uniqueness in the Gaussian case }

The following result was essentially established in \cite{AF}. We give below a slightly modified version with a sketch of the proof.

\begin{theorem}
\label{AF-u}
Assume that there exist $p>1$, $q>1$ such that 
$\|c\| \in L^p(\gamma)$, $\|Dc\|_{\mathcal{HS}} \in L^{q}(\gamma)$, $\mbox{\rm{div}}_{\gamma} c \in L^q(\gamma)$.

Then for every  $t_0>0$ there exists at most one solution to (\ref{main-eq})  satisfying $$\sup_{0 \le t \le t_0} \|\rho(t,\cdot)\|_{L^r(\gamma)} < \infty,$$ where $r \ge \max(p^*,q^*)$.

If, in addition, $d<\infty$ and $p \ge q$,  the assumption $\|Dc\|_{\mathcal{HS}} \in L^{q}(\gamma)$ can be replaced by $\| Dc\| \in L^q_{loc}(\gamma)$.
\end{theorem}
{\bf Sketch of the proof.} We discuss only the case $d<\infty$ (the proof for $d=\infty$ is almost the same). Fix a non-negative $C^{\infty}_0(\mathbb{R}^d)$-function $\varphi$.  Let $\rho$ be a solution to (\ref{main-eq}).
Set: $\hat{\rho} = \varphi \cdot \rho$ (this is important only for  $d <\infty$, since in this case we apply the local assumption). Take any $\psi \in C^{\infty}_0(\mathbb{R}^d)$ such that $\psi=1$ on $\mbox{supp}(\varphi)$ and set $\hat{c} = \psi \cdot c$. Note that $\hat{\rho}$ solves the following equation
$$
\frac{d}{dt} \hat{\rho} + \mbox{div}_{\gamma}(\hat{c}, \hat{\rho}) = \langle \nabla \varphi, \hat{c} \rangle \rho.
$$ 
Let us smoothen $\hat{\rho}$ with the Ornstein-Uhlenbeck semigroup: $ \rho^{\varepsilon} = e^{-\varepsilon} T_{\varepsilon} (\hat{\rho})$. One has
$$
\frac{d}{dt} {\rho}^{\varepsilon} + \mbox{div}_{\gamma}(\hat{c} \cdot \rho^{\varepsilon}) =
e^{-\varepsilon} r^{\varepsilon} (\hat{\rho}, \hat{c}) 
+ \rho^{\varepsilon} \mbox{div}_{\gamma} (\hat{c}) 
+ e^{-\varepsilon} T_{\varepsilon}( \langle \nabla \varphi, \hat{c} \rangle \rho),
$$ 
where $r^{\varepsilon}(v,b) = e^{\varepsilon} \langle b, \nabla(T_{\varepsilon}(v)) - T_{\varepsilon}(\mbox{div}_{\gamma}
(v \cdot b))$.

The uniqueness proof relies on the concept of the so-called renormalized solutions. Take a continuously differentiable globally Lipschitz function $v$ and using smoothness of 
$\rho^{\varepsilon}$ compute $\frac{d}{dt} v(\rho^{\varepsilon})$:
\begin{align*}
\frac{d}{dt}  v(\rho^{\varepsilon})  & + \mbox{div}_{\gamma}(\hat{c} \cdot v(\rho^{\varepsilon})) = \bigl[ v(\rho^{\varepsilon}) - \rho^{\varepsilon} v'(\rho^{\varepsilon})\bigr]\cdot \mbox{div}_{\gamma} (\hat{c}) 
\\&  + v'(\rho^{\varepsilon}) \bigl( e^{-\varepsilon} r^{\varepsilon} (\hat{\rho}, \hat{c}) 
+ \rho^{\varepsilon} \mbox{div}_{\gamma} (\hat{c}) 
+ e^{-\varepsilon} T_{\varepsilon}( \langle \nabla \varphi, \hat{c} \rangle \rho) \bigr).
\end{align*}
According to estimate (68) from \cite{AF} there exists $C= C(p,q)$ such that for $r = \max(p^*,q^*)$ and small values of $\varepsilon$ one has
$$
\|r^{\varepsilon}\|_{L^1(\gamma)} \le C \|\hat{\rho}\|_{L^r(\gamma)} ( \sqrt{\varepsilon} \|\hat{c}\|_{L^p(\gamma)} + \|\mbox{div}_{\gamma} \hat{c}\|_{L^q(\gamma)}  + \| (D \hat{c})^{sym} \|_{L^q(\gamma)}).
$$ 
It follows from the assumptions of this  theorem that the right-hand side is finite. In addition, $r^{\varepsilon} \to - \mbox{div}_{\gamma} (\hat{c}) \cdot \hat{\rho}$
in $L^1(\gamma)$ as $\varepsilon \to 0$ (Proposition 3.5 in \cite{AF}). 

Passing to the limit one obtains that
$$
\frac{d}{dt}  v(\hat{\rho})   + \mbox{div}_{\gamma}(\hat{c} \cdot v(\hat{\rho})) = \bigl[ v(\hat{\rho}) - \hat{\rho} v'(\hat{\rho})\bigr]\cdot \mbox{div}_{\gamma} (\hat{c}) 
  + v'(\hat{\rho}) ( \langle \nabla \varphi, \hat{c} \rangle \rho)
$$
in the distributional sense (i.e., $\rho$ is a {\it renormalized} solution). 
Assume that there exists two different solutions $\rho_1, \rho_2$ in $L^{r}(\gamma)$ with the same initial condition.
Applying this relation to the difference $\rho = \rho_1 - \rho_2$ and $v(t) = \max(0,t)$, we get
that for every $\varphi \in C^{\infty}_0(\mathbb{R}^d)$ one has
$
\frac{d}{dt} (\varphi \cdot \rho_{+}) + \mbox{div}_{\gamma}({c} \cdot ({\rho})_{+} \varphi) =  \langle \nabla \varphi, c \rangle \rho_{+}.
$ Finally,
$$
\frac{d}{dt} \rho_{+} + \mbox{div}_{\gamma}({c} \cdot ({\rho})_{+}) =  0
$$
in the distributional sense. Clearly, $\frac{d}{dt} \int \rho_{+} \ d \gamma =0$, hence $\int \rho_{+} \ d \gamma=0$ and $\rho=0$.$\Box$

\section{Examples of uniqueness}

In this section we study uniqueness problem for transport equations.
As in Theorem \ref{exist-th} we reduce the proof to the Gaussian case (see Theorem \ref{AF-u}).

Recall that 
$$
c  = DS(T) \cdot b(T).
$$
Since the assumption on the divergence  can be directly transferred, we need only to find some  sufficient conditions for
$$
\|c\| , \ \|D c \|_{HS} \in L^p(\gamma).
$$
One can try to apply the trivial operator norm estimate
$$
\|c\| \le  \| DS(T)\| \|b(T)\|.
$$

Let us stress, however, that operator norm estimates do not seem to be available in the case of triangular mappings
(unlike optimal transportation ones).
In spite of this let us give another estimate of $c$ which does not use operator norms.

\begin{lemma}
\label{20.01.2011}
For every  $1 \le p \le 2$ and $q \ge 1$ one has
\begin{align*}
\int \|c\|^p \ d\gamma 
\le  C(p,q) \Bigl(  \sup_i  \int \beta^{pq}_i \ d\nu \Bigr)^{\frac{1}{q}} \cdot
\Bigl[ \sum_{i}^{\infty}  
\Bigl( \int |b_i|^{pq^*}  \ d\nu \Bigr)^{\frac{1}{2q^*}} \Bigr]^2.
\end{align*}
\end{lemma}
\begin{proof}
Trivially we have
$$
\|c\| \le \sum_{i=1}^{\infty} |b_i(T)| \cdot \| \partial_{e_i} S(T) \|.
$$
Applying the inequality
$$
\sum_i |a_i| \le \bigl( \sum_i |a_i|^{\frac{1}{q}} \bigr)^{q}
$$
which holds for every $q \ge 1$,
we get for every $ 1 \le p \le 2$
\begin{align*}
\|c\|^p  & = \Bigl(  \sum_{i=1}^{\infty} |b_i(T)| \cdot \| \partial_{e_i} S(T) \| \Bigr)^p 
\le
 \Bigl(  \sum_{i=1}^{\infty} |b_i(T)|^{\frac{p}{2}} \cdot \| \partial_{e_i} S(T) \|^{\frac{p}{2}} \Bigr)^2
\\& 
=
\sum_{i,j=1}^{\infty} 
|b_i(T)|^{\frac{p}{2}} \cdot |b_j(T)|^{\frac{p}{2}} \cdot \| \partial_{e_i} S(T) \|^{\frac{p}{2}}
 \cdot \| \partial_{e_j} S(T) \|^{\frac{p}{2}}.
\end{align*}
By the H{\"o}lder inequality and Proposition \ref{l2triest}
\begin{align*}
\int \|c\|^p \ d\gamma & \le 
\sum_{i,j=1}^{\infty}  \Bigl( \int
|b_i(T)|^{\frac{pq^*}{2}} \cdot |b_j(T)|^{\frac{pq^*}{2}}  \ d\gamma \Bigr)^{\frac{1}{q^*}}  \Bigl(
\int \| \partial_{e_i} S(T) \|^{\frac{pq}{2}}
 \cdot \| \partial_{e_j} S(T) \|^{\frac{pq}{2}}   \ d\gamma \Bigr)^{\frac{1}{q}}
\\&
\le 
\sum_{i,j=1}^{\infty}  
\Bigl( \int |b_i|^{pq^*}  \ d\nu \Bigr)^{\frac{1}{2q^*}} 
 \cdot 
\Bigl( \int |b_j|^{pq^*}   \ d\nu \Bigr)^{\frac{1}{2q^*}} 
\cdot
 \Bigl( \int \| \partial_{e_i} S\|^{pq}  \ d\nu \Bigr)^{\frac{1}{2q}}
\cdot
 \Bigl( \int \| \partial_{e_j} S\|^{pq}  \ d\nu \Bigr)^{\frac{1}{2q}}
\\&
\le  C(p,q) \Bigl(  \sup_i  \int \beta^{pq}_i \ d\nu \Bigr)^{\frac{1}{q}} \cdot
\Bigl[ \sum_{i}^{\infty}  
\Bigl( \int |b_i|^{pq^*}  \ d\nu \Bigr)^{\frac{1}{2q^*}} \Bigr]^2.
\end{align*}
\end{proof}

\subsection{Product case}

\begin{theorem}
\label{prodcaseth}
Assume that
$\nu$ is a product measure on $\mathbb{R}^{\infty}$
\begin{equation}
\label{prodnu}
\nu = \prod_{i=1}^{\infty} e^{-w_i(x_i)} \ dx_i.
\end{equation}
Assume that for some $\varepsilon>0, \delta>0, 1 < p \le 2, q >1$
\begin{itemize}
\item[1)]
$$t \omega'_i(t) \ge \frac{(-1+ \delta)\varepsilon}{{pq^*( pq^*+ \varepsilon)-\varepsilon}}
\ \ \mbox{for every $i$}
$$
and
$$\sup_i \int \beta^{pq^*+\varepsilon}_i  \ d\nu < \infty;$$
\item[2)]
 $\|Db\|_{HS} \in L^{p}(\nu)$ and 
$$
\sum_{i=1}^{\infty}  
\Bigl( \int |b_i|^{pq}  \ d\nu \Bigr)^{\frac{1}{q}} 
+
\sum_{i \ne j} 
\Bigl( \int |\partial_{x_j} b_i|^{pq}  \ d\nu \Bigr)^{\frac{1}{q}}
< \infty;
$$
\item[3)]
$\mbox{\rm{div}}_{\nu} (b) \in L^{p}(\nu)$.
\end{itemize}
Then $
 \|c\| , \ \|D c \|_{HS} \in L^p(\gamma)
$.  In particular,   for every $t_0>0$ there exists at most one solution to the  equation (\ref{main-eq})  satisfying $$\sup_{0 \le t \le t_0} \| \rho(t,\cdot)\|_{L^{p^*}(\nu)} < \infty.$$
\end{theorem}

\begin{proof}
First we check that the assumptions of  Lemma \ref{btoc} and Proposition \ref{gauss-ex} are satisfied.
This is clear except for the estimate $\sup_{0 \le t \le t_0} \| \rho \cdot b_i\|_{L^{(p q^*)^*}} < \infty$. 
To prove this we apply the H{\"o}lder inequality
$\| \rho \cdot b_i\|_{L^{(p q^*)^*}} = \| \rho \cdot b_i\|_{L^{pq/(pq-q+1)}} \le \|\rho\|_{L^{p^*}} \| b_i\|_{L^{pq}}$.

By Theorem \ref{AF-u} and Proposition \ref{gauss-ex} the problem is now reduced to the uniqueness problem in the Gaussian case.

Thus, it is sufficient to show that
$
 \|c\| , \ \|D c \|_{HS} \in L^p(\gamma)
$
for some $p>1$. 
Since we deal with a product measure, the transportation mapping has a simple structure
$$
T =(T_1(x_1), T_2(x_2), \cdots, T_n(x_n), \cdots ).
$$
Hence
$$
c_i = {\partial_{x_i} S_i}(T_i)b_i(T).
$$

We apply  Lemma \ref{20.01.2011}. Note that that in this case $\partial_{e_j} S_i =0$.  Taking this into account and 
following the proof of  Lemma \ref{20.01.2011} we can get a more precise estimate:
\begin{align*}
\int \|c\|^p  d \gamma & \le C(p,q)
\sum_{i=1}^{\infty}  
\Bigl( \int |b_i|^{pq}  \ d\nu \Bigr)^{\frac{1}{q}} 
\cdot
 \Bigl( \int \beta^{pq^*}_i  \ d\nu \Bigr)^{\frac{1}{q^*}} \\&
 \le  C(p,q)
\sup_i  \Bigl( \int \beta^{pq^*}_i  \ d\nu \Bigr)^{\frac{1}{q^*}}
\sum_{i=1}^{\infty}  
\Bigl( \int |b_i|^{pq}  \ d\nu \Bigr)^{\frac{1}{q}}  , \ \ \ 1<p<2.
\end{align*}

Let us estimate $DC$. Taking into account that $S_i$ and $T_i$ are reciprocal, one easily gets
$$
\partial_{x_i} c_i = \Bigl( \frac{\partial_{x_i x_i} S_i}{\partial_{x_i} S_i} \cdot b_i+ \partial_{x_i} b^i \Bigr) \circ T.
$$
and
$$
\partial_{x_j} c_i = \Bigl( \frac{\partial_{x_i} S_i}{\partial_{x_j} S_j} \partial_{x_j} b^i \Bigr) \circ T, \ \ i \ne j.
$$

One can estimate $\|(\partial_{x_i} c_i)\|$ in the same way as $\|c\|$ and applying  Proposition \ref{Siij} one gets
\begin{align*}
\int \|(\partial_{x_i} c_i)\|^p \ d\gamma &
\le C(p) \Bigl[  \int \| (\partial_{x_i} b^i) \|^p \ d\nu 
+ \sup_i \Bigl( \int \bigl( \frac{\partial_{x_i x_i} S_i}{\partial_{x_i} S_i}\bigr)^{pq^*}  \ d\nu \Bigr)^{\frac{1}{q^*}}
\sum_{i=1}^{\infty}  
\Bigl( \int |b_i|^{pq}  \ d\nu \Bigr)^{\frac{1}{q}} \Bigr]
\\&
\le C(p,q) \Bigl[ \int \| (\partial_{x_i} b^i) \|^p \ d\nu 
+ \sup_i \Bigl( \| \beta_i\|_{L^{pq^* +\varepsilon}(\nu) }\Bigr)^{p}
\sum_{i=1}^{\infty}  
\Bigl( \int |b_i|^{pq}  \ d\nu \Bigr)^{\frac{1}{q}} \Bigr].
\end{align*}

Similarly
$$
\int \| B \|_{\mathcal{HS}}^p \ d\gamma
\le
\sup_{i,j}  \Bigl(  \int \Bigl[ \frac{\partial_{x_i} S_i}{\partial_{x_j} S_j}\Bigr]^{pq^*} \ d \nu \Bigr)^{\frac{1}{q^*}} 
\sum_{i \ne j} 
\Bigl( \int |\partial_{x_j} b_i|^{pq}  \ d\nu \Bigr)^{\frac{1}{q}},
$$
where
 $B_{i,j} = \partial_{x_i} c_i$, $i \ne j$,  $B_{i,i}=0$.

It remains to estimate 
\begin{align*}
\int \Bigl[ \frac{\partial_{x_i} S_i}{\partial_{x_j} S_j}\Bigr]^{pq^*} \ d \nu &
\le
\Bigl[ \int \bigl[ \partial_{x_i} S_i \bigr]^{{pq^*}+\varepsilon} \ d \nu\Bigr]^{\frac{{pq^*}}{{pq^*}+\varepsilon}}
\Bigl[ \int \Bigl[ \frac{1}{\partial_{x_j} S_j}\Bigr]^{\frac{{pq^*} (\varepsilon+{pq^*})}{\varepsilon}} \ d \nu \Bigr]^{\frac{\varepsilon}{{pq^*}+\varepsilon}}
\\&
= 
\Bigl[ \int \bigl[ \partial_{x_i} S_i \bigr]^{{pq^*}+\varepsilon} \ d \nu\Bigr]^{\frac{{pq^*}}{{pq^*}+\varepsilon}}
\Bigl[ \int \Bigl[ {\partial_{x_j} T_j}\Bigr]^{\frac{{pq^*} (\varepsilon+{pq^*})}{\varepsilon}} \ d \gamma \Bigr]^{\frac{\varepsilon}{{pq^*}+\varepsilon}}.
\end{align*}
According to Proposition  \ref{1d-pow-trans} and Proposition \ref{l2triest} we get that the right-hand side is finite if 
$
| \beta_i |_{L^{{pq^*}+\varepsilon} (\nu) }  < \infty
$
and $ t \omega'_i(t) \ge \frac{(-1+ \delta)\varepsilon}{pq^*(\varepsilon + pq^*) - \varepsilon}.
$
The proof is complete.
\end{proof}

\begin{corollary}
Let $\nu$ be as in (\ref{prodnu})
satisfying $t w'_i(t) \ge 0$. 
Assume that 
\begin{itemize}
\item[1)]
$$
\sup_{i} \int |w'_i(t)|^N e^{-w_i(t)} dt < \infty
$$
for every $N>1$:
\item[2)] for some $1<p \le 2$ and $q > 1$ one has $(\partial_{x_i} b_i ) \in L^p(\nu)$;
\item[3)]
$\mbox{\rm{di}v}_{\nu} (b) \in L^{p}(\nu)$
and
$$
\sum_{i=1}^{\infty}  
\Bigl( \int |b_i|^{pq}  \ d\nu \Bigr)^{\frac{1}{q}} 
+
\sum_{i \ne j} 
\Bigl( \int |\partial_{x_j} b_i|^{pq}  \ d\nu \Bigr)^{\frac{1}{q}}
< \infty.
$$
\end{itemize}Then $
 \|c\| , \ \|D c \|_{HS} \in L^p(\gamma)
$.  In particular,    for every $t_0>0$ there exists at most one solution to the  equation (\ref{main-eq})  satisfying $$\sup_{0 \le t \le t_0} \| \rho(t,\cdot)\|_{L^{p^*}(\nu)} < \infty.$$
\end{corollary}

\subsection{Gibbs measures}

In this section we prove uniqueness for the measures described in Example \ref{Gibbs}.
More generally, we will assume:

\

{ \bf Assumption (A)}:
There exist smooth functions $V_i(x_i), W_{i,j}(x_i,x_{j})$ such that
\begin{equation}
\label{log-der}
\beta_i = V'_i(x_i) + \sum_{j=1}^{\infty} \partial_{x_i} W_{i,j}
\end{equation}
and there exists $N_0 \ge 1$ such that $W_{i,j}=0$ if $|i-j| > N_0$.

Clearly, in this case the corresponding mapping $S$ has a special structure
$$
S(x_1, \cdots, x_n, \cdots) = (S_1(x_1), S_2(x_1, x_2), \cdots, S_{n}(x_{n-N_0}, \cdots, x_{n-1}, x_n), \cdots ).
$$

\begin{example}
Let  us consider a Gibbs measure $\nu = ''e^{-H} dx''$ with Hamiltonian
$$
H = \sum_{i=1}^{\infty} V_i(x_i) + \sum_{i,j=1}^{\infty} W_{i,j}(x_i,x_{j}).
$$
Under the assumptions of   Remark \ref{Gibbs}  there exists a unique  Gibbs measure $\nu$ satisfying (\ref{log-der}),
as explained above.
\end{example}

We will also need the following 1-dimensional version of the Caffarelli contraction theorem 
(which holds true for optimal transport mappings in any dimension, see  \cite{Caf}, \cite{Kol-contr}).
{ Note, however, that in the one-dimensional case the proof  is elementary and relies on explicit formulas.}

\begin{theorem}
Let $T : \mathbb{R} \to \mathbb{R}$ be the canonical  increasing mapping pushing forward a probability measure $e^{-V(x)} \ dx$ onto a
probability measure $e^{-W(x)} \ dx$. Assume that $V$ and $W$ are twice continuously differentiable and
$
V'' \le C, \ W'' \ge K.
$
Then $T$ is Lipschitz satisfying  $T' \le \sqrt{\frac{C}{K}}$
\end{theorem}

We recall that a probability measure $\mu$  on $\mathbb{R}^d$ is called log-concave if it has the form $e^{-V} \cdot \mathcal{H}^{k}|_{L}$, where $\mathcal{H}^k$ is the $k$-dimensional Hausdorff
measure, $k \in \{0,1, \cdots, d\}$, $L$ is an affine subspace,  and { $V : L \to (-\infty,+\infty]$ is a lower semicontinuous} convex function. We call a measure $\mu$ uniformly log-concave if $\frac{1}{Z} e^{K |x|^2} \cdot \mu$ is a log-concave measure for
some $K>0$ and a suitable renormalization factor $Z$. It is well-known (C. Borell) that the projections and conditional measures of log-concave measures are log-concave. The same holds for uniformly log-concave measures.
We can extend this notion to the infinite-dimensional case. Namely, we call a probability measure $\mu$ on a locally convex space $X$ log-concave (uniformly log-concave with $K>0$)  if its images  $\mu \circ l^{-1}$, $l \in X^*$, under all linear continuous
functionals are  all log-concave (uniformly log-concave with  $K>0$).  
We will also use the fact that the (one-dimensional) conditional measures of uniformly log-concave measures  are uniformly log-concave with the same constant.

\begin{theorem}
Assume that assumption 
{\bf (A)} is satisfied and
\begin{itemize}
\item[1)]
 for every $n \ge 1$
$$
\sup_{i} \int \Bigl( |\beta_i|^n +\|\nabla \beta_i\|^n  \Bigr) \ d \mu < \infty;
$$
\item[2)]
$\nu$ is uniformly log-concave;
\item[3)] 
for some $ 1 < p_0 \le  2, q_0>1$
$$
\sum_{i=1}^{\infty}  
\Bigl( \int |b_i|^{p_0 q_0}  \ d\nu \Bigr)^{\frac{1}{2q_0}} 
< \infty ;$$
\item[4)]  for some $1<p<2$ 
$$
\sum_{j,k=1}^{\infty} k^{\frac{p}{2}} \Bigl( \int   | \partial_{x_k} b_{j}|^{\frac{4p}{2-p}}  \ d \nu \Bigr)^{\frac{2-p}{8}} < \infty.
$$
\end{itemize}
Then $ \|c\| , \ \|D c \|_{HS} \in L^q(\gamma)$ for some $q>1$. If, in addition, 
$\mbox{\rm{div}}_{\nu} (b) \in L^{q}(\nu)$, then  for every $t_0$ there exists at most one solution
to (\ref{main-eq})  satisfying  $\sup_{0 \le t \le t_0 }\|\rho(t, \cdot)\|_{L^{q^*}(\nu)} < \infty$.
\end{theorem}

\begin{remark}
a) According to results of \cite{AKRT} there exist probability measures satisfying the assumptions of the theorem (see  Remark \ref{Gibbs}). 
In particular,  1) is automatically satisfied for these measures.  

b) We believe that the factor $k^{\frac{p}{2}}$ in 4) can be removed. This factor arise just because we deal with triangular transportations and the quantity
$\nabla T_k$ is difficult to control (see Lemma  \ref{2011.07.24} ).  One can get a better control applying optimal transportation mappings. { Unfortunately, the existence
of such a mapping in the infinite dimensional case for mutually singular measures is not a trivial problem.
 Some counter-examples to existence of infinite-dimensional optimal mappings are known for quite simple situations (see \cite{Kol2012}).
 It is explained in \cite{Kol2012} that the optimal transportation mapping may not exist if one of the measures is ergodic and another is not.
}

c) The assumption of uniformly log-concavity of the Gibbs measure $\nu$ can be expressed in terms of the potentials $V_i, W_{ij}$. Note that it is sufficient to require
the uniformly log-concavity of the approximations.
\end{remark}

\begin{proof}
It is  easy to check that $b$, $\nu$, and $\rho$ satisfy the assumptions of Lemma \ref{btoc},  Lemma \ref{core-trans}, and Proposition \ref{gauss-ex}.
Thus the problem is reduced to the uniqueness problem in the Gaussian case.
 We have to show that $ \|c\| , \ \|D c \|_{HS} \in L^q(\gamma)$ for some $q>1$. 
The first part follows from Lemma 
\ref{20.01.2011}. The second part follows from Lemmata \ref{2011.07.24},   \ref{2011.07.24(2)}. 
Indeed, note that 
\begin{equation}
\label{ds-db}
Dc = \bigl( DS \cdot Db \cdot (DS)^{-1} \bigr) \circ T
+ \Bigl( \sum_{i=1}^{\infty} b_i \cdot (DS_{x_i}) (DS)^{-1} \Bigr) \circ T.
\end{equation}

 In the same way as in Lemma \ref{20.01.2011} and applying Lemma  \ref{2011.07.24(2)}
we get the desired estimate for $\sum_{i=1}^{\infty} b_i   (DS_{x_i}) (DS)^{-1}$. Note that $\frac{1}{\partial_{x_i} S_i} = \partial_{x_i} T_i(S)$ is bounded by $\frac{1}{\sqrt{K}}$
as a one-dimensional optimal mapping of a Gaussian measure onto a uniformly log-concave measure by the Caffarelli theorem. Here we use the fact that the corresponding  conditional measures  are uniformly log-concave.

We apply Lemma  \ref{2011.07.24} to estimate $DS \cdot Db \cdot (DS)^{-1}$. To complete the proof we need to estimate  $\int | \nabla T_k|^2 \ d \gamma$. Indeed, since $\nu$ is 
uniformly log-concave, we can apply Remark \ref{int-ident}. We get $K \int |\partial_{x_i} T|^2 \ d \nu \le \int x^2_i \ d \gamma$.
Hence
$$
\int \| \nabla T_k\|^2 \ d \gamma  = \int \sum_{i=1}^k (\partial_{x_i} T_k)^2 \ d\gamma
\le \sum_{i=1}^k  \int  \| \partial_{x_i} T\|^2 \ d\gamma \le \frac{1}{K} \sum_{i=1}^k  \int  x^2_i \ d\gamma = \frac{k}{K}.
$$
\end{proof}

\begin{lemma}
\label{2011.07.24}
For every $1 < p < 2$ there exists $C$, depending on $p, N_0$, and 
$ \sup_{i} \int |\beta_i|^{\frac{4p}{2-p}}  \ d \nu $ such that
$$
\int \| DS \cdot Db \cdot (DS)^{-1} \|^p_{\mathcal{HS}} \le  C \Bigl[ \sum_{j,k=1}^{\infty} \Bigl( \int   | \partial_{x_k} b_{j}|^{\frac{4p}{2-p}}  \ d \nu \Bigr)^{\frac{2-p}{8}}  \Bigl( \int \| \nabla T_k\|^2 \ d \gamma  \Bigr)^{\frac{p}{4}}\Bigr]^2.
$$
\end{lemma}
\begin{proof}
  We estimate
$
\| DS \cdot Db \cdot (DS)^{-1} \|^2_{\mathcal{HS}} =  \sum_{i=1}^{\infty} \| DS \cdot Db \cdot (DS)^{-1} \cdot e_i \|^2.
$

For simplicity set
$ M = DS \cdot DB$, $L = (DS)^{-1}$. Then
\begin{align*}
\int & \Bigl( \sum_{i=1}^{\infty} \| M L e_i\|^2\Bigr)^{\frac{p}{2}} \ d \nu
\le 
\int \Bigl( \sum_{i=1}^{\infty} \Bigl(  \sum_{k=1}^{\infty}  |  L_{i,k}|  \|M \cdot e_k\| \Bigr)^2\Bigr)^{\frac{p}{2}} \ d \nu
\\&
\le
\int \Bigl[ \sum_{i=1}^{\infty}   \sum_{k,j=1}^{\infty} |  L_{i,k}|  |  L_{i,j}|  \|M \cdot e_k\| \|M \cdot e_j\| \Bigr]^{\frac{p}{2}} \ d \nu
 \le
 \int    \sum_{k,j=1}^{\infty} \sum_{i=1}^{\infty} |  L_{i,k}|^{\frac{p}{2}}  |  L_{i,j}|^{\frac{p}{2}}  \|M \cdot e_k\|^{\frac{p}{2}} \|M \cdot e_j\|^{\frac{p}{2}} \ d \nu
\\& \le 
  \sum_{k,j=1}^{\infty}  \Bigl(  \int  \|M \cdot e_k\|^{\frac{p}{2-p}} \|M \cdot e_j\|^{\frac{p}{2-p}} \ d \nu\Bigr)^{\frac{2-p}{2}} 
\Bigl( \int \sum_{i=1}^{\infty}  |  L_{i,k}|  |  L_{i,j}|  \ d \nu\Bigr)^{\frac{p}{2}}
\\&
\le 
  \sum_{k,j=1}^{\infty} \Bigl[  \Bigl(  \int  \|M \cdot e_k\|^{\frac{2p}{2-p}} \ d\nu \Bigr)^{\frac{2-p}{4}}  \Bigl( \int \|M \cdot e_j\|^{\frac{2p}{2-p}} \ d \nu\Bigr)^{\frac{2-p}{4}} 
\Bigl( \int \sum_{i=1}^{\infty}  |  L_{i,k}|^2  \ d \nu\Bigr)^{\frac{p}{4}} \Bigl( \int \sum_{i=1}^{\infty}  | L_{i,j}|^2 \ d \nu\Bigr)^{\frac{p}{4}}\Bigr] 
\\&
=
\Bigl(   \sum_{k=1}^{\infty}  \Bigl(  \int  \|M \cdot e_k\|^{\frac{2p}{2-p}} \ d\nu \Bigr)^{\frac{2-p}{4}} \Bigl( \int \sum_{i=1}^{\infty}  |  L_{i,k}|^2  \ d \nu\Bigr)^{\frac{p}{4}}    \Bigr)^2.
\end{align*}
Furthermore,
$$
 \sum_{i=1}^{\infty}  \int | L_{i,k}|^2 \ d \nu =  \sum_{i=1}^{\infty}  \int | (DT)_{i,k}|^2 \ d \gamma 
= \sum_{i=1}^{k}  \int | \partial_{x_i} T_k|^2 \ d \gamma  =  \int \| \nabla T_k\|^2 \ d \gamma 
$$
and
\begin{align*}
 \int & \|M \cdot e_k\|^{\frac{2p}{2-p}} \ d\nu
=  \int  \Bigl(  \sum_{i=1}^{\infty} ( \sum_{r=0}^{N_0} \partial_{x_{i}} S_{i-r} \cdot \partial_{x_k} b_{i-r} )^2 \Bigr)^{\frac{p}{2-p}} \ d\nu
\\& \le C(p,N_0) \int \sum_{i=1}^{\infty}  \Bigl(  \sum_{r=0}^{N_0} |\partial_{x_i} S_{i-r}|^{\frac{2p}{2-p}}  |\partial_{x_k} 
b_{i-r}|^{\frac{2p}{2-p}}  \Bigr) \ d \nu
\\& \le C(p,N_0)  \sum_{i=1}^{\infty}   \sum_{r=0}^{N_0} \Bigl( \int |\partial_{x_i} S_{i-r}|^{\frac{4p}{2-p}}  \ d \nu   \Bigr)^{1/2}
\Bigl( \int   |\partial_{x_k} b_{i-r}|^{\frac{4p}{2-p}}  \ d \nu \Bigr)^{1/2} \\&
\le
C(p,N_0)    \sup_{i} \Bigl( \sum_{j=i-N_0}^{i}  \Bigl( \int |\partial_{x_i} S_{j}|^{\frac{4p}{2-p}}  \ d \nu   \Bigr)^{1/2}\Bigr) \sum_{j=1}^{\infty} \Bigl( \int   |\partial_{x_k} b_{j}|^{\frac{4p}{2-p}}  \ d \nu \Bigr)^{1/2}
\\&
\le C \sum_{j=1}^{\infty} \Bigl( \int   | \partial_{x_k} b_{j}|^{\frac{4p}{2-p}}  \ d \nu \Bigr)^{1/2}.
\end{align*}
In the last estimate we apply Proposition \ref{lptriest} and the special structure of $S$ .

Finally, we obtain
$$
\int \| DS \cdot Db \cdot (DS)^{-1} \|^p_{\mathcal{HS}} \le  C \Bigl[ \sum_{j,k=1}^{\infty} \Bigl( \int   |b_{j,k}|^{\frac{4p}{2-p}}  \ d \nu \Bigr)^{\frac{2-p}{8}}  \Bigl( \int \| \nabla T_k\|^2 \ d \gamma  \Bigr)^{\frac{p}{4}}\Bigr]^2.
$$
\end{proof}

\begin{lemma} \label{2011.07.24(2)} Assume that for every $p \ge 1$
$$
\sup_{i} \int \Bigl( |\beta_i|^p + \|\nabla \beta_i\|^p  + \frac{1}{(\partial_{x_i} S_i)^p} \Bigr) \ d \mu \le C(p).
$$
Then 
 $\sup_i \int \|  (DS_{x_i}) (DS)^{-1} \|^p_{\mathcal{HS}} \  d\nu \le D$
with $D$ depending on $p$ and $N_0$.
\end{lemma}
\begin{proof}
Due to  the special structure of $S$ the matrix $D S_{x_i}$ only  has a finite number of non-zero entries. The result now  follows immediately from Remark \ref{lptriest} and
Propositions \ref{Siij}, \ref{Sjmi}.
\end{proof}

\section{Existence and uniqueness of the associated flows}

{
According to  \cite{flow2009}, the flows of vector fields, associated with $b$, can be defined in the following way.  

\begin{definition}
\label{flow-def}
The Borel mapping $X: [0,T] \times \mathbb{R}^{\infty} \to \mathbb{R}^{\infty}$
is a $L^r$-regular flow associated to $b$ if
\begin{itemize}
\item[(i)]
for $\nu$-a.e. $x$ the map $t \to \|b(X(t,x))\|$ belongs to $L^1(0,T)$ and
\begin{equation}
\label{Xflow}
X(t,x) = x + \int_0^t b(X(s,x)) \ ds,
\end{equation}
where the integral understood in the weak sense, namely
$$
\langle X(t,x)-x, e_i \rangle = \int_0^t \langle b(X(s,x)),e_i \rangle \ ds, \ \forall e_i
$$
\item[(ii)] 
for all $t \in [0,T]$ the law of $X$ under $\nu$ is absolutely continuous with respect to $\nu$
with a density $\rho_t$ in $L^r(\nu)$ and $\sup_{t \in [0,T]} \| \rho_t\|_{L^r(\nu)} < \infty$.
\end{itemize}
\end{definition}

Existence and uniqueness results for the flows in the Gaussian case have been obtained by 
Ambrosio and Figalli in \cite{flow2009}. Following our approach one can transfer their solution for the vector field $c$ to the 
solution for $b$. Namely, let us consider a $L^r$-regular flow $Y$ generated by $c$:
\begin{equation}
\label{flowGauss}
\dot{Y}(t,y) = c(Y(t,y)), \ \ Y(0,y)=y.
\end{equation}
The transfered flow $X(t,x)$ can be naturally defined as follows:
\begin{equation}
\label{flowtransfer}
X(t,x) = S^{-1}(Y(t,S(x))).
\end{equation}

It is natural to expect that $X$ is the desired flow associated with $b$.
In this paper we  prove neither existence nor uniqueness of  such a flow. Any result of this kind  is not an immediate consequence of
the Ambrosio-Figalli result because   $X$ is obtained from $Y$ via the composition with the non-smooth mapping $S$. This makes the justification non-trivial.
We also stress that  sufficient conditions for the uniqueness and existence are essentially different  similarly to the case of the continuity equation for densities. In addition, the reader should take into account
our weaker sufficient conditions for uniqueness from Section 2. 

Nevertheless,  let us indicate some ideas how to prove that given a flow $Y$ associated with $c$ the flow $X$ given by formula (\ref{flowGauss}) is associated with $b$ in the sense 
of Definition \ref{flow-def} (this provides the existence part. For uniqueness one has to prove vice versa: given a flow $X_t$ associated with $b$ show that $T^{-1} \circ X_t \circ T$ is a flow associated with $c$). 

\begin{enumerate}
\item[1)] The part (ii) is obvious by the change of variables formula.
\item[2)] 
Consider a mapping $U$ of the type $U(y) = y - P_n(y) + U_n(P_n y)$  where $P_n y= \sum_{i=1}^d y_i e_i$ is the projection on the first $n$
coordinates and $U_n : \mathbb{R}^n \to \mathbb{R}^n$ is a  $n$-dimensional globally Lipschitzean diffeomorphism. 
Then the formula
$Y(t,y) = Y(s,y) + \int_s^t  c(Y(\tau,y))\ d \tau$ (see Proposition 4.11 and formula (52) in \cite{flow2009}) and the chain rule imply together the following relation
$$
U(Y(t,y)) = U(Y(s,y)) + \int_s^t  DU \cdot c(U \circ Y(\tau,y))\ d \tau, \ \ \mbox{for $\gamma$-a.e. $y$}.
$$
\item[3)]
Construct approximations $U_n \to T$ of $T$ by smooth mappings with the properties declared in 2)  and pass to the limit in the identity
$\int \Bigl| U_n(Y(t,y)) - U_n(Y(s,y)) - \int_s^t  DU_n \cdot c(U_n \circ Y(\tau,y))\ d \tau \Bigr| d \gamma =0$. One obtains
$$T(Y(t,y)) - T(Y(s,y)) - \int_s^t  DT \cdot c(T \circ Y(\tau,y))\ d \tau =0$$ for $\gamma$-a.e. $y$. For this step we need bounds for the Sobolev norms of $T$ and $U_n$.
Applying this identity to $y = T^{-1} x$ we get the desired identity for $X_t= T \circ Y_t \circ T^{-1}$. 
\end{enumerate}
}

\section{Appendix: finite-dimensional case and optimal transportation} 

The purpose of this section is to add a few comments about the connection to optimal transportation and to show that our results in this paper, when applied to the special case
of finite dimensions, i.e. on $\mathbb{R}^d$, lead to new results not covered by the existing literature (\cite{Ambrosio04}, \cite{Ambrosio05}, \cite{DPL1}, \cite{DPL2}, \cite{BL}).

Instead of triangular mappings one can also apply the optimal transportation mappings. For a detailed account on optimal transportation see \cite{BoKo2011}, \cite{Vill}.
In this case the available a-priori estimates are essentially better in many respects. For instance, 
there exist $L^p$-estimates on operator norms of $DT$ which do not depend on dimension (see \cite{Kol2010}, \cite{BoKo2011}). 
Unfortunately, this approach has certain disadvantages: 1) unlike the triangular mappings, the optimal transportation mappings do not have an explicit form
and the a-priori estimates for them are usually hard to prove, 2) the existence problem for  optimal transportation mappings in infinite dimensions is 
 solved in sufficient generality only for the case when  the measures $\mu$ and $\nu$  have a finite Kantorovich distance $W_2(\mu,\nu)$
 (see \cite{FU1},  \cite{Kol04})).  If $\mu=\gamma$ is Gaussian, this limitation means  that basically
we  should restrict ourselves to measures which are absolutely continuous with respect to $\gamma$, i.e. $\nu = g \cdot \gamma$ and, moreover, have finite entropy,  that is $\int g \log g \ d \gamma < \infty$.
{The finiteness of the Kantorovich distance in this case follows from the Talagrand's transportation inequality
$$
\frac{1}{2} W_2^2(\gamma, g \cdot \gamma) \le \int g \log g \ d \gamma.
$$
We stress that $\int g \log g \ d \gamma < \infty$ is the most natural and simple sufficient condition for finiteness of the Kantorovich distance on the Wiener space.
To our knowlegde, the finiteness of $W_2(\gamma,\nu)$ does not necessary imply that $\nu$ is absolutely continuous with respect to $\gamma$.
}

\begin{remark}
Some new existence  results on optimal transportation of certain Gibbs measures are obtained in \cite{Kol2012}. 
These results together with estimates from \cite{Kol2010}, \cite{BoKo2011} can be used to obtain infinite-dimensional uniqueness/existence statements for the case where $\nu$ is uniformly log-concave.
But we don't consider this approach in this paper.
\end{remark}

We assume in the rest of this section that  $d<\infty$. We consider  the  optimal transportation mapping $T$  pushing forward
the standard Gaussian measure $\gamma$ onto  $\nu = e^{-W} dx$.
In particular,  $T$ has the form $T = \nabla \Psi$, where $\Psi$ is a convex function.
The inverse mapping $S=T^{-1}$ is optimal too and has the form $S = \nabla \Phi$, where $\Phi$ is the  convex conjugate to $\Psi$.

The drifts $c$ and $b$ are related in the same way as above
$$
c = DS(T) \cdot b(T) = D^2 \Phi(T) \cdot b(T).
$$

Let us illustrate how our methods work in the finite-dimensional case.

\begin{theorem} {\bf (Uniqueness)}
Assume that $W$ is locally H{\"o}lder,
$\|Db\| \in L^2_{loc}(\nu)$, $|b| \in L^4(\nu)$, $|\nabla W| \in L^4(\nu)$. Then $\|c\| \in L^2(\gamma), \|D^2 c\| \in L^2_{loc}(\gamma)$ and 
for every $t_0$ and every fixed initial condition $\rho_0 \in L^2(\nu)$ there exists at most one solution
to (\ref{main-eq})  satisfying  $\sup_{0 \le t \le t_0 }\|\rho(t, \cdot)\|_{L^{2}(\nu)} < \infty$.
\end{theorem}

\begin{remark}
1) Note that we do not need any bounds on the {\it second} derivatives of $W$.

2) The assumption of H{\"o}lder continuity is made only to assure high enough local integrability (even boundedness) of $\| D^2 \Phi \| \cdot \|(D^2 \Phi )^{-1}\|$.
We believe that this can be achieved in some other (more efficient) way without H{\"o}lder continuity. We note, however, that Sobolev estimates for optimal transportation  in $W^{1,p}_{loc}$ with big $p$ are hard to prove.
See in this respect the recent paper \cite{KoTi} and the references therein.

3) Some estimates applied in the proof are valid in the infinite-dimensional setting.
\end{remark}

\begin{proof}
First we note that the second partial derivative $\partial_{e v} \Phi$ of $\Phi$  is  locally H{\"o}lder for any vectors $e,v \in \mathbb{R}^d$, because both measures 
$\nu$ and $\gamma$ have Lebesgue densities which are locally H{\"o}lder. This follows from the well-known results of Caffarelli \cite{Caf-90} (see 
\cite{Kol-hoelder} for technical improvements for unbounded domains). Clearly, the same holds for $D^2 \Phi$.  Let us apply 
Proposition \ref{gauss-ex}. We need to show that $\|c\| \in L^{2}(\gamma)$, $\|Dc \|\in L^2_{loc}(\gamma)$. 
 
For $\|c\|$ one has the trivial  estimate
$$
\|c\| \le \| D\Phi (T)\| \cdot \|b(T)\|,
$$
which implies
$$
2  \int \|c\|^2 \ d \nu \le \int \|D^2 \Phi\|^4 d \nu + \int \|b\|^4 \ d \nu.
$$
By Theorem 6.1 from \cite{Kol2010} assumption $\|\nabla W\| \in L^4(\nu)$ implies that $\|D^2 \Phi \| \in L^4(\nu)$. Hence $\|c\| \in L^2(\nu)$.

Let us estimate $Dc$:
$$
Dc = \bigl[ D^2 \Phi \cdot Db \cdot (D^2 \Phi)^{-1} \bigr] (\nabla \Psi)
+    \bigl[ ( D^2 \Phi)_b (D^2 \Phi)^{-1} \bigr] (\nabla \Psi),
$$
where  for brevity we set
$$
( D^2 \Phi)_b =  \sum_{i=1}^{d} b_i \cdot \partial_{x_i} (D^2 \Phi).
$$
Using the standard
$HS$-norm inequalities
$$
\|AB\|_{HS} \le \|A\| \cdot \|B\|_{HS}, \   \  \
\|BA\|_{HS} \le \|A\| \cdot \|B\|_{HS} ,
$$
applied to an arbitrary $A$ and symmetric $B$, we get
$$
\|Dc\|_{HS}    \le \Bigl( \mathcal{L}  \|Db\|_{HS} +  \mathcal{L}^{1/2}
  \cdot \| (D^2 \Phi)^{-1/2} ( D^2 \Phi)_b (D^2 \Phi)^{-1/2} \|_{HS}
\Bigr) (\nabla \Psi),
$$
where
$
 \mathcal{L}  = \|D^2 \Phi\| \cdot \| (D^2 \Phi)^{-1}\|.
$

Since $ \|D^2 \Phi\|,  \| (D^2 \Phi)^{-1}\|$ are  locally bounded functions, we only need  to locally estimate 
$$
 \ \| (D^2 \Phi)^{-1/2} ( D^2 \Phi)_b (D^2 \Phi)^{-1/2} \|_{HS};
$$
To this end we apply the following inequality
\begin{align}
\label{wbb}
\int & \langle D^2 W \cdot b,b \rangle \eta  \  d \nu  
\ge \int \| D^2 \Phi \cdot b \|^2  \eta  \  d \nu
+ \int \langle (D^2 \Phi)_b \cdot b  , (D^2 \Phi)^{-1}\nabla \eta \rangle  \ d\nu
\\&
\nonumber
+ 2  \int \mbox{\rm Tr} \bigl(   (D^2 \Phi)_b   \cdot  Db  \cdot  (D^2 \Phi)^{-1} \bigr)  \eta  \ d \nu
+ \int \| (D^2 \Phi)^{-1/2} (D^2 \Phi)_b (D^2 \Phi)^{-1/2}\|^2_{{HS}} \eta \ d\nu,
\end{align}
proved in Lemma 7.1 \cite{Kol2010}. Here $\eta \in C^{\infty}_0(\mathbb{R}^d)$ with $\int \Bigl\|\frac{\nabla \eta}{\eta}\Bigr\|^p \eta \ d \nu < \infty$ for a sufficiently big $p$.

Since we do not assume existence of $D^2 W$, we apply the integration by parts formula to get rid of this:
\begin{align*}
\int \langle D^2 W \cdot b,b \rangle \eta & \  d \nu  
= \int \langle \nabla W, b \rangle ^2 \eta \ d \nu - \int   \langle \nabla W, b \rangle \mbox{div}_{\mu} b \cdot  \eta \ \ d \nu
- 
\int \langle \nabla W  , Db \cdot b \rangle \eta  \  d \nu 
\\& +  \int \langle \nabla W, b \rangle  \langle  \nabla W,  \nabla \eta \rangle \ d \nu.
\end{align*}
Clearly, the right-hand side is finite by the Cauchy inequality.

The elementary estimates 
\begin{align} 
\label{elem-est1}
\int \langle (D^2 \Phi)_b \cdot b  , (D^2 \Phi)^{-1}\nabla \eta \rangle \ d\nu & \le \varepsilon \int \| (D^2 \Phi)^{-1/2} (D^2 \Phi)_b (D^2 \Phi)^{-1/2} \|^{2}_{\mathcal{HS}} \eta \ d \nu
\\& \nonumber + C(d,\varepsilon) \int \|\mathcal{L}\| |b|^2 \frac{\|\nabla \eta\|^2}{\eta}  \ d \nu,
 \end{align}
and
\begin{align}
\label{LDb}
2  \int \mbox{\rm Tr} \bigl(   (D^2 \Phi)_b   \cdot  Db  \cdot  (D^2 \Phi)^{-1} \bigr)  \eta  \ d \nu &
\le \varepsilon \int \| (D^2 \Phi)^{-1/2} (D^2 \Phi)_b (D^2 \Phi)^{-1/2} \|^{2}_{{HS}} \eta \ d \nu
\\& 
\nonumber
+ C(\varepsilon) \int  \|\mathcal{L}\| \|Db \|^2_{{HS}} \eta d \nu. 
\end{align}
imply that  $\| (D^2 \Phi)^{-1/2} (D^2 \Phi)_b (D^2 \Phi)^{-1/2}\|^2_{{HS}} \in L^2_{loc}(\nu)$,  hence  $\|D c\| \in L^2_{loc}(\gamma)$ . 
\end{proof}

The following theorem is formulated in a "dimension-free" manner. This means in particular that this formulation makes sense in the infinite-dimensional setting.
Actually, we believe that an appropriate generalization  of the  theorem always holds in the infinite-dimensional case  provided the corresponding optimal transportation map  of $\nu$ to $\gamma$ does exist.

\begin{theorem} {\bf (Existence and uniqueness)}
\label{logconcave}
Let $\nu$ be a probability measure with Lebesgue density on $\mathbb{R}^d$.
Assume that $\nu$  is a  uniformly log-concave measure, so,  in particular, there exists $K>0$ such that $D^2 W \ge K \cdot \rm{Id}$ in the interior of the support of $\nu$.   Assume that  $W$ belongs to $W^{2,p}(\nu)$ for some $p \ge 1$ and, moreover, 
$$\partial_{x_i} W \in L^{2p}(\nu), \ \ 1 \le i \le d, \ \ \| D^2 W \|^{p} \in L^1(\nu).$$
In addition, we assume that
$$
\int \langle D^2 W b, b \rangle \ d \nu <\infty, \ b_i \in L^{\frac{4p}{2p-1}}(\nu), \ \ 1 \le i \le d,
$$
and
$$ \mbox{\rm{div}}_{\nu} b \in L^2(\nu),  \ \|Db\|^{\frac{4p}{2p-1}}_{HS} \in L^1(\nu).
$$
Then 
\begin{itemize}
\item[1)] under the additional assumption that $e^{\varepsilon (\rm{div}_{\nu} b)_{-} } \in L^1(\nu)$,
there exists a solution to
 (\ref{main-eq}) for any initial $\rho_0 \in L^{2+\varepsilon}(\nu)$ satisfying $\sup_{0 \le t \le t_0 }\|\rho(t, \cdot)\|_{L^{2}(\nu)} < \infty$;
\item[2)]
any two solutions to
 (\ref{main-eq})  satisfying  $\sup_{0 \le t \le t_0 }\|\rho(t, \cdot)\|_{L^{2}(\nu)} < \infty$ with the same initial condition $\rho_0$ coincide.
\end{itemize}
\end{theorem}
\begin{proof}
The proof follows the same line as the proof of the previous theorem.
According to the Caffarelli's theorem  $\| (D^2 \Phi)^{-1} \| \le \frac{1}{\sqrt{K}}$. According to a result of \cite{Kol2010}
$\int \| D^2 \Phi\|^{2p} \ d \nu \le \int \|D^2 W\|^p \ d \nu < \infty$.
This implies that $\mathcal{L} \in L^{2p}(\nu)$.

We use the estimates from the proof of the previous theorem.  
By H{\"o}lder's inequality $\int  \|\mathcal{L}\| \|Db \|^2_{{HS}} \  d \nu < \infty$. Applying that $\int \langle D^2 W b, b \rangle \ d \nu <\infty$ we get from (\ref{wbb}) and (\ref{LDb}) 
that
$$
 \int \| (D^2 \Phi)^{-1/2} (D^2 \Phi)_b (D^2 \Phi)^{-1/2} \|^{2}_{{HS}}  d \nu < \infty, \ \ \int \|D^2 \Phi \cdot b\|^2 \ d \nu < \infty,
$$
and, finally $\|c\|^2, \|Dc\|^2_{HS} \in L^1(\nu)$. The uniqueness statement follows easily from Proposition \ref{gauss-ex} with the help of   H{\"o}lder's inequality.

For proving existence, we cannot use  Theorem \ref{exist-th} directly, because assumptions 1) and 4) are not fulfilled in general. Note, however, that we need 1) and 4) only to apply Proposition \ref{nu-ex}.
But the statement of  Proposition \ref{nu-ex} holds trivially in our case because of the uniform bound $\| (D^2 \Phi)^{-1} \| \le \frac{1}{\sqrt{K}}$ (see the  proof of Lemma \ref{core-trans}). 
\end{proof}

\begin{example}
\label{fin-dim-ex}
Let
$$
\nu = \rho(x) \ dx,  \ \ \ \rho(x) =  C \prod_{i=1}^{d} e^{-(\frac{1}{x_i^2} + \frac{x_i^2}{2})} I_{\{x_i > 0\}}.
$$ 
 Note that $\nu = e^{-W}$ is uniformly log-concave and $\|\nabla W\|, \| D^2 W\|$ belongs to $L^p(\nu)$ for any $p>0$. 
Fix an  arbitrary $q>0$ and set $b = \frac{x}{\|x\|^q}$. It is easy to check that $\rm{div}_{\nu} b$ is bounded from below and the other assumptions on $b$ are satisfied.
Hence for every $\rho_0 \in L^{2+\varepsilon}(\nu)$ there exists a unique short-time (even long-time) solution to  (\ref{main-eq}).

Note that if $q$ is big, then $b$ is not a BV function with respect  to Lebesgue measure. This makes inapplicable the finite-dimensional theory from 
\cite{Ambrosio05}. 
\end{example}

\end{document}